\documentclass{article}

\usepackage{amsmath}
\usepackage{amssymb}
\usepackage{amsfonts}
\usepackage{amsthm}
\usepackage[left=1in,right=1in,top=1in,bottom=1in]{geometry}
\usepackage[backref=page]{hyperref}
\usepackage[width=14cm,format=plain,labelfont=sf,bf,small,textfont=sf,up,small,justification=justified,labelsep=period]{caption}
\usepackage{graphicx}
\usepackage{framed}
\usepackage{xcolor}
\usepackage{tabularx}
\usepackage{tikz}
\usepackage{algorithm}
\usepackage{algorithmic}
\usepackage{bm}
\usepackage[mathscr]{euscript}
\usepackage[outerbars,color]{changebar}
\usepackage{booktabs}

\hypersetup{
    colorlinks=true,       
    linkcolor=blue,          
    citecolor=blue,        
    filecolor=blue,      
    urlcolor=blue           
}

\makeatletter
\def\iddots{\mathinner{\mkern1mu\raise\p@
\vbox{\kern7\p@\hbox{.}}\mkern2mu
\raise4\p@\hbox{.}\mkern2mu\raise7\p@\hbox{.}\mkern1mu}}
\makeatother

\newcommand{\hadprod}{\circ} 

\newcommand{\cX}{\mathcal X}

\newcommand{\cF}{\mathcal F}

\newcommand{\cL}{\mathcal L}
\newcommand{\cM}{\mathcal M}

\newcommand{\RR}{\mathbb{R}}
\newcommand{\CC}{\mathbb{C}}
\newcommand{\QQ}{\mathbb{Q}}
\newcommand{\ZZ}{\mathbb{Z}}

\newcommand{\x}{\boldsymbol{x}}
\newcommand{\psd}{\succeq}

\definecolor{darkgreen}{rgb}{0,0.8,0}

\DeclareMathOperator{\supp}{supp}

\DeclareMathOperator{\Tr}{{\bf Tr}}

\DeclareMathOperator{\conv}{conv}
\DeclareMathOperator{\diag}{diag}

\let\Im\relax 
\DeclareMathOperator{\Im}{Im}
\let\Re\relax 
\DeclareMathOperator{\Re}{Re}

\DeclareMathOperator{\TPol}{TPol}
\DeclareMathOperator{\indiam}{indiam}

\renewcommand{\imath}{\boldsymbol{i}}

\newcommand{\Rot}{\textup{Rot}} 
\newcommand{\Freq}[1]{\ZZ_N} 

\newcommand{\PAR}{\text{PAR}} 

\newcommand{\Gon}{\mathcal{X}} 

\renewcommand{\S}{\mathbf{S}} 
\newcommand{\HH}{\mathbf{H}} 

\newcommand{\isom}{\cong} 

\newcommand{\mb}[1]{\bm{#1}}

{\begin{list}{}%
         {\setlength{\leftmargin}{1cm}\setlength{\rightmargin}{1cm}}%
         \item[]\small%
}
{\end{list}}

\renewcommand{\tilde}[1]{\widetilde{#1}}

\theoremstyle{plain}
\newtheorem{prop}{Proposition}

\newtheorem{lem}{Lemma}
\newtheorem{thm}{Theorem}

\newtheorem{cor}{Corollary}

\newtheorem{question}{Question}

\theoremstyle{definition}
\newtheorem{defn}{Definition}
\newtheorem{example}{Example}

\theoremstyle{remark}
\newtheorem*{rem}{Remark}

\makeatletter
\newif\if@borderstar
 \def\bordermatrix{\@ifnextchar*{%
 \@borderstartrue\@bordermatrix@i}{\@borderstarfalse\@bordermatrix@i*}%
}
 \def\@bordermatrix@i*{\@ifnextchar[{\@bordermatrix@ii}{\@bordermatrix@ii[()]}}
\def\@bordermatrix@ii[#1]#2{%
\begingroup
\m@th\@tempdima8.75\p@\setbox\z@\vbox{%
 \def\cr{\crcr\noalign{\kern 2\p@\global\let\cr\endline }}%
 \ialign {$##$\hfil\kern 2\p@\kern\@tempdima & \thinspace %
 \hfil $##$\hfil && \quad\hfil $##$\hfil\crcr\omit\strut %
 \hfil\crcr\noalign{\kern -\baselineskip}#2\crcr\omit %
 \strut\cr}}%
 \setbox\tw@\vbox{\unvcopy\z@\global\setbox\@ne\lastbox}%
 \setbox\tw@\hbox{\unhbox\@ne\unskip\global\setbox\@ne\lastbox}%
 \setbox\tw@\hbox{%
 $\kern\wd\@ne\kern -\@tempdima\left\@firstoftwo#1%
 \if@borderstar\kern2pt\else\kern -\wd\@ne\fi%
 \global\setbox\@ne\vbox{\box\@ne\if@borderstar\else\kern 2\p@\fi}%
 \vcenter{\if@borderstar\else\kern -\ht\@ne\fi%
 \unvbox\z@\kern -\if@borderstar2\fi\baselineskip}%
 \if@borderstar\kern -2\@tempdima\kern2\p@\else\,\fi\right\@secondoftwo#1 $%
}\null \;\vbox{\kern\ht\@ne\box\tw@}%
\endgroup
}
\makeatother

\title{Equivariant semidefinite lifts of regular polygons}

\author{Hamza Fawzi \and James Saunderson \and Pablo A. Parrilo\thanks{The authors are with
    the Laboratory for Information and Decision Systems, Department of
    Electrical Engineering and Computer Science, Massachusetts
    Institute of Technology, Cambridge, MA 02139. Email:
    \texttt{\{hfawzi,jamess,parrilo\}@mit.edu}.}}

\renewcommand\footnotemark{}

\date{September 9, 2014}

\begin{document}

\maketitle

\begin{abstract}
Given a polytope $P\subset \RR^n$, we say that $P$ has a positive semidefinite lift (psd lift) of size $d$ if one can express $P$ as the linear projection of an affine slice of the positive semidefinite cone $\S^d_+$. 
If a polytope $P$ has symmetry, we can consider \emph{equivariant psd lifts},
i.e.\ those psd lifts that respect the symmetry of $P$. One of the 
simplest families of polytopes with interesting symmetries are regular
polygons in the plane, which have played an important role in the study of \emph{linear programming lifts} (or 
\emph{extended formulations}). In this paper
we study equivariant psd lifts of regular polygons.

We first show that the standard Lasserre/sum-of-squares hierarchy for the regular $N$-gon requires exactly $\lceil N/4 \rceil$ iterations and thus yields an equivariant psd lift of size linear in $N$.
In contrast we show that one can construct an equivariant psd lift of the regular $2^n$-gon of size $2n-1$, which is exponentially smaller than the psd lift of the sum-of-squares hierarchy. Our construction relies on finding a \emph{sparse} sum-of-squares certificate for the facet-defining inequalities of the regular $2^n$-gon, i.e., one that only uses a small (logarithmic) number of monomials. Since any equivariant LP lift of the regular $2^n$-gon must have size $2^n$, this gives the first example of a polytope with an exponential gap between sizes of equivariant LP lifts and equivariant psd lifts.
Finally we prove that our construction is essentially optimal by showing that any equivariant psd lift of the regular $N$-gon must have size at least logarithmic in $N$.
\end{abstract}


\section{Introduction}

\subsection{Preliminaries}


Semidefinite programming is the problem of minimizing (or maximizing) a linear function subject to linear matrix inequalities. The feasible set of a semidefinite program is known as a \emph{spectrahedron} and corresponds to an affine slice of the cone of positive semidefinite matrices. An important question that has attracted a lot of attention in optimization is to give representations of convex sets as feasible sets of semidefinite programs.

In this paper we are interested in \emph{lifted} semidefinite representations. We say that a convex set $C$ has a \emph{positive semidefinite lift} (psd lift) if it can be written as the linear projection of a spectrahedron. More formally, we have the following definition of psd lift:

\begin{defn}
\label{def:psdlift}
\cite{gouveia2011lifts}
Let $C$ be a convex set in $\RR^n$, and let $\S^d_+$ be the cone of real symmetric positive semidefinite matrices. We say that $C$ has a $\S^d_+$-lift, or a \emph{psd lift of size $d$}, if there exists a linear map $\pi : \S^d \rightarrow \RR^n$ and an affine subspace $L \subset \S^d$ such that:
\[ C = \pi(\S^d_+ \cap L). \]
\end{defn}

An interesting question that has gained a lot of interest recently is, for a given polytope $P$, to characterize the size of the \emph{smallest psd lift} of $P$. This quantity, known as the \emph{psd rank} of $P$, was introduced in \cite{gouveia2011lifts} and studied in e.g., \cite{fiorini2012linear,gouveia2013polytopes,briet2013existence,fawzi2014equivariant}. Such psd lifts are interesting in practice when the size of the psd lift is much smaller than the number of facets of $P$ (which is the size of the trivial representation of $P$). Indeed, if $P$ has a psd lift of size $d$, then one can formulate any linear optimization problem over $P$ as a semidefinite program of size $d$.

The definition of a psd lift given here was first formulated in \cite{gouveia2011lifts} and is the generalization of the notion of \emph{LP lift} (also called \emph{extended formulation}) to the case of semidefinite programming. The definition of an LP lift of size $d$ is similar to that of a psd lift except that the psd cone $\S^d_+$ is replaced by $\RR^d_+$ (see Appendix \ref{app:equivariantLPlifts} for the formal statement of the definition).
 The size of the smallest LP lift of a polytope $P$ is called the \emph{LP extension complexity}. An important question in the area of lifted representations of polytopes is to know whether there exist polytopes with large gaps between sizes of LP lifts and psd lifts. The following open question is taken from \cite{psdranksurvey}:
\begin{question}
\label{q:gap}
Find a family of polytopes that exhibits a large (e.g. exponential) gap between its
psd rank and LP extension complexity.
\end{question}
One of the results of this paper shows that regular polygons give an example of such a gap when we restrict to lifts that respect symmetry, in a sense that we now make precise.

\paragraph{Equivariant lifts} In many situations, the polytope $P$ of interest has certain symmetries. The symmetries of a polytope $P \subset \RR^n$ are the geometric transformations that leave $P$ invariant: more precisely if $G$ is a group linearly acting on $\RR^n$, we say that $P$ is invariant under the action of $G$ if $g\cdot x \in P$ for any $x \in P$ and $g \in G$. 
 In \cite{fawzi2014equivariant}, we studied so-called \emph{equivariant psd lifts} of polytopes, which are psd lifts that respect the symmetry of a polytope $P$. Intuitively, a psd lift $P=\pi(\S^d_+ \cap L)$ respects the symmetry of $P$ if any transformation $g \in G$ that leaves $P$ invariant can be \emph{lifted} to a transformation $\Phi(g)$ of $\S^d$ that leaves the cone $\S^d_+$ and the subspace $L$ invariant and such that the following natural \emph{equivariance} condition holds: $\pi(\Phi(g) Y) = g \pi(Y)$ for all $Y \in \S^d_+ \cap L$. Since the transformations that leave the psd cone $\S^d_+$ invariant are precisely the congruence transformations (i.e., transformations of the form $Y \mapsto RYR^T$ where $R \in GL_d(\RR)$, cf. \cite[Theorem 9.6.1]{tuncel2000potential}), the transformation $\Phi(g)$ is required to have the form $\Phi(g):Y \mapsto \rho(g) Y \rho(g)^T$ where $\rho:G\rightarrow GL_d(\RR)$. This leads to the following definition of \emph{equivariant psd lift} from \cite{fawzi2014equivariant}:
\begin{defn}
\label{def:equivariance}
Let $P \subseteq \RR^n$ be a polytope and assume $P$ is invariant under the action of some group $G$. Let $P = \pi(\S^d_+ \cap L)$ be a $\S^d_+$-lift of $P$, where $L \subset \S^d$ is an affine subspace of the space of real symmetric $d\times d$ matrices. The lift is called \emph{$G$-equivariant} if there exists a homomorphism $\rho : G \rightarrow GL_d(\RR)$ such that:
\begin{itemize}
\item[(i)] The subspace $L$ is invariant under congruence transformations by $\rho(g)$, for all $g \in G$, i.e.:
\begin{equation}
  \label{eq:Linvariance}
 \rho(g) Y \rho(g)^T \in L \quad \forall g \in G, \; \forall Y \in L.
\end{equation}
\item[(ii)] The following equivariance relation holds:
\begin{equation}
 \label{eq:linequivariance}
 \pi\left(\rho(g) Y \rho(g)^T\right) = g\pi(Y) \quad \forall g \in G, \; \forall Y \in \S^d_+ \cap L.
\end{equation}
\end{itemize}
\end{defn}

In \cite{fawzi2014equivariant} we studied equivariant psd lifts of a general class of symmetric polytopes known as \emph{regular orbitopes}. An \emph{orbitope} \cite{sanyal2011orbitopes} is a polytope of the form $P = \conv(G\cdot x_0)$ where $G\cdot x_0$ is the orbit of $x_0 \in \RR^n$ under the action of a finite group $G$. Furthermore, we say that $P$ is a \emph{regular orbitope} if for any distinct elements $g\neq g'$ of $G$, we have $g\cdot x_0 \neq g' \cdot x_0$  (i.e., the points $\{g\cdot x_0, g \in G\}$ are in one-to-one correspondence with the elements of the group $G$). In \cite{fawzi2014equivariant} we established a connection between equivariant psd lifts of such polytopes and sum-of-squares certificates for its facet-defining inequalities. This connection allowed us to prove exponential lower bounds on sizes of equivariant psd lifts for two families of polytopes, namely the cut polytope and the parity polytope.

Note that when working with LP lifts (i.e., lifts with the cone $\RR^d_+$) one can also give a natural definition of \emph{equivariant LP lift} (also known as symmetric LP lift in the literature). The definition of an equivariant LP lift is similar to that of an equivariant psd lift, except that the action by congruence transformations $Y\mapsto \rho(g)Y \rho(g)^T$ is replaced by a permutation action $y \mapsto \Phi(g) y$ where $\Phi:G\rightarrow \mathfrak{S}_d$ is a homomorphism from $G$ to the permutation group on $d$ elements\footnote{The automorphism group of $\S^d_+$ consists of congruence transformations, whereas the automorphism group of $\RR^d_+$ is the permutation group on $d$ elements $\mathfrak{S}_d$.}  (see Appendix \ref{app:equivariantLPlifts} for the formal statement).

Equivariant LP lifts of various polytopes have been studied before in \cite{yannakakis1991expressing,kaibel2012symmetry,pashkovich2009tight,gouveia2011lifts} and it was shown that the requirement of equivariance can affect the size of the lift: several examples have been provided of polytopes with an exponential gap between sizes of LP lifts and equivariant LP lifts. One of the simplest such examples are regular $N$-gons in the plane which are known to have an LP lift of size $\log N$ \cite{ben2001polyhedral}, and yet any equivariant LP lift must have size at least $N$ when $N$ is a power of a prime \cite{gouveia2011lifts} (see Appendix \ref{app:equivariantLPlifts} for more details).

Further results about lifts of non-regular polygons were obtained in \cite{fiorini2012extended} where it was shown that generic $N$-gons have LP extension complexity at least $\sqrt{2N}$ (where generic means that the coordinates of the vertices are algebraically independent over $\QQ$). For psd lifts of $N$-gons much less is known. The only asymptotic lower bound on the \emph{psd rank} of $N$-gons is $\Omega\left(\sqrt{\frac{\log N}{\log \log N}}\right)$ which come from quantifier elimination theory \cite{gouveia2011lifts,gouveia2013worst}. It is also known that generic $N$-gons have psd rank at least $(2N)^{1/4}$ \cite{gouveia2013worst}.

\subsection{Summary of contributions} 
In this paper we propose to study equivariant psd lifts of regular polygons. Our contribution is threefold:
\begin{enumerate}
\item To obtain an equivariant psd lift of the regular polygon, one way is to use the Lasserre/sum-of-squares hierarchy \cite{lasserre2009convex,gouveia2010theta}. Our first contribution is to show that the sum-of-squares hierarchy for the regular $N$-gon requires exactly $\lceil N/4 \rceil$ iterations.
 The lower bound of $\lceil N/4 \rceil$ seems to be known in the community, though not written explicitly anywhere. Our main contribution here is to show that the $\lceil N/4 \rceil$'th iteration is exact (the previously known upper bound was $\lceil N/2 \rceil-1$). We prove this new upper bound by exploiting the fact that the regular $N$-gon is a \emph{$\lceil N/2 \rceil$-level} polytope and by showing that in some cases ---and in the particular case of regular polygons--- $k$-level polytopes only require $\lceil k/2 \rceil$ levels of the sum-of-squares hierarchy, instead of the previously known bound of $k-1$ \cite[Theorem 11]{gouveia2012convex}. The results developed here are of independent interest and can be applied to other $k$-level polytopes.
\item The second contribution of the paper is to give an explicit construction of an \emph{equivariant psd lift} of the regular $2^n$-gon of size $2n-1$.
 The main feature of our construction
 is that it is \emph{equivariant} (with respect to the full dihedral group), unlike the LP lift of Ben-Tal and Nemirovski \cite{ben2001polyhedral} which is not equivariant. It was actually shown in \cite[Proposition 3]{gouveia2011lifts} that any equivariant LP lift of the regular $N$-gon must have size at least $N$ when $N$ is a prime or a power of a prime (cf. Appendix \ref{app:equivariantLPlifts} for more details). 
Our construction thus gives an exponential gap between sizes of equivariant psd and linear programming lifts and thus gives an answer to a restricted version of Question \ref{q:gap}, where the restriction is to the case of equivariant lifts. 
Also note that the size of our construction is exponentially smaller than the lift obtained from the Lasserre/sum-of-squares hierarchy, which has size $1+2^{n-1}$.
Finally, another property of our lift is that it can be described using rational numbers only, whereas the LP lift \cite{ben2001polyhedral} involves irrational numbers.

\item Lastly, we prove that our equivariant lift is optimal by proving a lower bound on the sizes of equivariant psd lifts of the regular $N$-gon. In fact we show that any equivariant Hermitian psd lift of the regular $N$-gon must have size at least $\ln(N/2)$. The main ingredient in this part is a new result establishing a lower bound on the \emph{sparsity} of polynomials that arise in any sum-of-squares certificate of the facet-defining inequalities of the regular $N$-gon.
\end{enumerate}


Table \ref{tbl:polygons} summarizes the known bounds concerning LP/PSD lifts of the regular $2^n$-gon in both the equivariant and non-equivariant cases.

\begin{table}[ht]
  \centering
   \begin{tabular}{lp{6.8cm}p{7.3cm}}
   \toprule
      & Equivariant & Non-equivariant\\ \midrule
   LP & Lower bound: $2^n$ \cite{gouveia2011lifts}\newline Upper bound: $2^n$ (trivial) &  Lower bound: $n$ 
\cite{goemans2009smallest}\newline Upper bound: $2n+1$\cite{ben2001polyhedral} \\ \hline
SDP & {\bf Lower bound: $(\ln 2) (n-1)$ (Theorem \ref{thm:lbhermitian})\newline Upper bound: $2n-1$ 
(Section \ref{sec:construction})} & Lower bound: $\Omega\left(\sqrt{\frac{n}{\log n}}\right)$ \cite{gouveia2011lifts,gouveia2013worst} \newline Upper bound: $2n-1$ 
(Section \ref{sec:construction})  \\
\bottomrule
\end{tabular}
  \caption{\label{tbl:polygons} Bounds on the size of the smallest LP/PSD lifts for the regular $2^n$-gon in both the equivariant and non-equivariant cases. The main contributions of this paper are highlighted in bold.}
\end{table}


\subsection{Notation and statement of results}

We now describe the results of the paper more formally and introduce some notation and terminology that will be useful later.
The regular $N$-gon we consider in this paper has vertices
\[ \Gon_N = \{(\cos \theta_i,\sin \theta_i), i=1,\dots,N\} \;\; \text{ where } \;\; \theta_i = \frac{(2i-1) \pi}{N}. \]
Figure \ref{fig:Ngon} shows a picture for $N=7$.

\begin{figure}[ht]
  \centering
  \includegraphics[width=6cm]{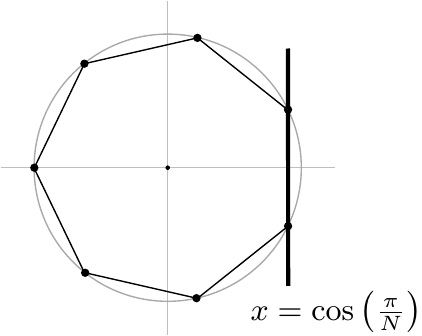}
  \caption{The regular $7$-gon.
}
  \label{fig:Ngon}
\end{figure}
The ``first'' (rightmost) facet of the regular $N$-gon is defined by the linear inequality $x \leq \cos(\pi/N)$. A main concern in this paper is to study certificates of nonnegativity of this linear function on $\cX_N$ using sum-of-squares, i.e., to certify that 
\begin{equation}
 \label{eq:facetineq}
 \cos(\pi/N) - x \geq 0 \quad \forall (x,y) \in \Gon_N.
\end{equation}
As we show later, such sum-of-squares certificates are the key to obtain equivariant psd lifts of the regular $N$-gon (cf. Section \ref{sec:soslift} for details).

Let $I \subseteq \RR[x,y]$ be the vanishing ideal of the vertices of the $N$-gon, i.e., $I$ is the set of polynomials that vanish on $\cX_N$. Note that $I$ is generated by the polynomials $\Re[(x+iy)^N+1]$ and $\Im[(x+iy)^N+1]$ where $\Re$ and $\Im$ indicate real and imaginary parts.
Let $\cF(N,\RR)=\RR[x,y]/I$ be the space of real polynomials on the vertices of the $N$-gon. The space $\cF(N,\RR)$ decomposes into a direct sum according to degree:
\begin{equation}
\label{eq:decompFN}
\cF(N,\RR) = \TPol_0(N) \oplus \dots \oplus \TPol_{\lfloor N/2 \rfloor}(N),
\end{equation}
where $\TPol_k(N)$ is the space of homogeneous polynomials of degree $k$ on the $N$-gon. For $0 < k < N/2$, the space $\TPol_k(N)$ is two-dimensional and is spanned by the functions $c_k$ and $s_k$:
\begin{equation}
 \label{eq:cksk}
 c_k(x,y) = \Re[(x+iy)^k], \quad s_k(x,y) = \Im[(x+iy)^k].
\end{equation}
Note that we can rewrite functions $c_k$ and $s_k$ using cosines and sines as follows (we can identify functions on $\cX_N$ as functions on the set of angles $\{\theta_1,\dots,\theta_N\}$):
\[ c_k(\theta) = \cos(k\theta), \quad s_k(\theta) = \sin(k\theta) \qquad \forall \theta \in \{\theta_1,\dots,\theta_N\}. \]
Observe that $s_0 = 0$ and, when $N$ is even, $c_{N/2} = 0$ (since $\cos((N/2) \theta_i) = \cos( (2i-1)\pi/2 ) = 0$ for all $i = 1,\dots,N$). Thus we have:
\begin{equation}
\label{eq:dimTPolk}
 \dim \TPol_k(N) = \begin{cases} 1 & \text{ if } k = 0 \text{ or } k = N/2\\ 2 & \text{ else. } \end{cases}
\end{equation}
We verify that $\sum_{k=0}^{\lfloor N/2 \rfloor} \dim \TPol_k(N) = \dim\cF(N,\RR) = N$.
Observe that a decomposition of any function $f \in \cF(N,\RR)$ in terms of the functions $c_k$ and $s_k$ essentially corresponds to a real \emph{discrete Fourier transform}. Thus we will often refer to functions in $\TPol_k(N)$ as functions with ``frequency'' $k$.

The Lasserre/sum-of-squares hierarchy for the regular $N$-gon seeks to certify the facet inequality \eqref{eq:facetineq} using sum-of-squares of polynomials of degree $\leq k$. The $k$'th level of the sum-of-squares hierarchy is called \emph{exact} if there exist polynomials $h_i$ with $\deg h_i \leq k$ (i.e., $h_i \in \TPol_0(N) \oplus \dots \oplus \TPol_k(N)$) such that:
\begin{equation}
\label{eq:cert-ksos}
 \cos(\pi/N) - x = \sum_{i} h_i(x,y)^2 \quad \forall (x,y) \in \Gon_N.
\end{equation}
The smallest $k$ for which a certificate of the form \eqref{eq:cert-ksos} exists (with $\deg h_i \leq k$) is called the \emph{theta-rank} of the $N$-gon, in reference to the terminology on theta-bodies \cite{gouveia2010theta}. In Section \ref{sec:soshierarchy} we prove the following theorem:
\begin{thm}
\label{thm:thetarank}
The theta-rank of the regular $N$-gon is exactly $\lceil N/4 \rceil$.
\end{thm}
In terms of psd lifts, Theorem \ref{thm:thetarank} means that the psd lift of the regular $N$-gon obtained from the sum-of-squares hierarchy has size \[ \dim (\TPol_0(N)\oplus \dots \oplus \TPol_{\lceil N/4 \rceil}(N)) = 1 + 2 \lceil N/4 \rceil. \] 

In this paper we show that one can actually obtain an equivariant psd lift that is substantially smaller than the one produced by the sum-of-squares hierarchy. The main idea here is to look for a sum-of-squares certificate of the form \eqref{eq:cert-ksos} where the functions $h_i$ are \emph{sparse} i.e., they are supported only on a few monomials (with potentially high degree). More precisely, we are looking for a certificate \eqref{eq:cert-ksos} where the functions $h_i$ live in a subspace $V$ of $\cF(N,\RR)$ of the form:
\begin{equation}
 \label{eq:subspaceV}
 V = \bigoplus_{k \in K} \TPol_k(N)
\end{equation}
where $K$ is a \emph{sparse} set (i.e., $K$ is not necessarily an interval $K=\{0,1,\dots,k\}$). If such a certificate exists, we say that $\ell$ has a sum-of-squares certificate with frequencies in $K$. Note that any subspace $V$ of the form \eqref{eq:subspaceV} is \emph{invariant} under the action of the dihedral group. Namely, if $f \in V$ and $\tau$ is any element of the dihedral group, then we have $\tau \cdot f \in V$ where $(\tau \cdot f)(x,y) := f(\tau^{-1} (x,y))$. Thus if a certificate of the form \eqref{eq:cert-ksos} exists where the functions $h_i$ are in the subspace $V$, then one automatically obtains a sum-of-squares certificate for all the other facets of the regular $N$-gon (since all the facets can be obtained from the facet $\ell = \cos(\pi/N) - x$ by rotation). One can then show that this leads to an equivariant psd lift of the regular $N$-gon of size $\dim V \leq 2|K|$.
%

In Section \ref{sec:construction}, we show that one can obtain sum-of-squares certificates for the facets of the $2^n$-gon where the set $K$ consists of powers of two. This constitutes the key result that leads to the equivariant lift of the $2^n$-gon of logarithmic size:
\begin{thm}
Consider the facet-defining linear function of the regular $2^n$-gon $\ell = \cos(\pi/2^n) - x$. Then $\ell$ admits a sum-of-squares certificate with frequencies in 
\[ K = \{0\} \cup \{2^i, i=0,\dots,n-2\}. \]
More precisely, there exist functions $h_k \in \TPol_0(2^n)\oplus \TPol_{2^k}(2^n)$ for $k=0,1,\ldots,n-2$ such that:
\begin{equation}
\label{eq:soscert2}
\ell = \sum_{k=0}^{n-2} h_k^2
\end{equation}
where \eqref{eq:soscert2} is an equality in $\cF(N,\RR)$ (i.e., an equality modulo the ideal of the $2^n$-gon).
\end{thm}
An immediate corollary of this sum-of-squares certificate is the following psd lift of the regular $2^n$-gon over $(\S^3_+)^{n-1}$:
\begin{thm}
\label{thm:equivariantliftintro}
The regular $2^{n}$-gon is the set of points $(x_0,y_0)\in \RR^2$ such that
there exist real numbers $x_1,y_1,\dots,x_{n-2},y_{n-2},y_{n-1}$ satisfying:
\begin{align*}
\begin{bmatrix} 1 & x_{k-1} & y_{k-1}\\x_{k-1} & \frac{1+x_k}{2} & \frac{y_k}{2}\\y_{k-1} &\frac{y_k}{2} & \frac{1-x_{k}}{2}\end{bmatrix} \in \S_+^3
\quad
\text{for $k=1,2,\ldots,n-2$}\quad
\text{and}\quad
\begin{bmatrix} 1 & x_{n-2} & y_{n-2}\\x_{n-2} & \frac{1}{2} & \frac{y_{n-1}}{2}\\y_{n-2} & \frac{y_{n-1}}{2} & \frac{1}{2}\end{bmatrix} \in \S_+^3.
 \end{align*}
 \end{thm}
One can also use the certificate \eqref{eq:soscert2} to obtain a psd lift of the regular $N$-gon with a ``single block'' (i.e., no Cartesian products) over $\S^{2n-1}_+$. We refer to Section \ref{sec:smalllift} for more details on this.


We have described a general way to obtain equivariant psd lifts of the regular $N$-gon, by looking for certificates of nonnegativity of $\ell$ using sum-of-squares in a certain invariant subspace $V$. 
To obtain lower bounds on the size of equivariant psd lifts, we need to understand the structure of all possible equivariant psd lifts of the $N$-gon.
Our structure theorem from \cite{fawzi2014equivariant} shows that any equivariant psd lift of the regular $N$-gon must be of the sum-of-squares form described above. More precisely the results of \cite{fawzi2014equivariant} imply:
\begin{thm}
\label{thm:structureR}
Consider the facet-defining linear function of the regular $N$-gon $\ell = \cos(\pi/N) - x$. 
Assume that the regular $N$-gon has a psd lift of size $d$ that is equivariant with respect to $\Rot_N$, the subgroup of the dihedral group consisting of rotations. Then there exists a $\Rot_N$-invariant subspace $V$ of $\cF(N,\RR)$ with $\dim V \leq 2d$ such that:
\[ \ell = \sum_{i} h_i^2 \quad \text{ where } \quad h_i \in V. \]
\end{thm}
Since the subspaces $\TPol_k(N), k \in \{0,\dots,\lfloor N/2 \rfloor\}$ are the only irreducible subspaces of $\cF(N,\RR)$ under the action of $\Rot_N$, it is not hard to see that any $\Rot_N$-invariant subspace $V$ of $\cF(N,\RR)$ must have the form:
\[ V = \bigoplus_{k \in K} \TPol_k(N) \]
where $K \subseteq \{0,\dots,\lfloor N/2 \rfloor\}$.
Thus, using Theorem \ref{thm:structureR}, the problem of studying equivariant psd lifts of the regular $N$-gon reduces to the problem of studying sum-of-squares certificates of  $\ell = \cos(\pi/N) - x$ with functions $h_i$ having frequencies in a set $K$.  
Our main result is to give a lower bound on the size of such sets $K$:
\begin{thm}
\label{thm:lbsosvalid}
Let $\ell = \cos(\pi/N) - x$ and assume that we can write
\begin{equation}
 \label{eq:ellcert3}
 \ell = \sum_{i} h_i^2 \quad \text{ where } \quad h_i \in \bigoplus_{k \in K} \TPol_k(N)
\end{equation}
for some set $K \subseteq \{0,\dots,\lfloor N/2 \rfloor\}$. Then necessarily $|K| \geq \ln(N/2)/2$.
\end{thm}
Combining Theorems \ref{thm:structureR} and \ref{thm:lbsosvalid} we get that any $\Rot_N$-equivariant psd lift of the regular $N$-gon must have size at least $\Omega(\log N)$.
 Our proof of Theorem \ref{thm:lbsosvalid} proceeds in two steps. In the first step, we give necessary conditions in terms of the ``geometry'' of a set $K$ for a sum-of-squares certificate \eqref{eq:ellcert3} to exist: we show that if the elements in $K$ can be \emph{clustered} in a certain way then a sum-of-squares certificate of $\ell$ of the form \eqref{eq:ellcert3} is not possible. In the second step we propose an algorithm to cluster any given set $K$, and we prove that our algorithm finds a valid clustering whenever the set $K$ is small enough, i.e., whenever $|K| < \ln(N/2)/2$. 
\begin{rem}
When proving Theorems \ref{thm:structureR} and \ref{thm:lbsosvalid} in Section \ref{sec:lb} we actually consider the more general Hermitian sum-of-squares certificates (instead of real sum-of-squares certificates) which are more convenient to work with. We refer the reader to Section \ref{sec:lb} for more details on this.
\end{rem}

\paragraph{Organization} 
The paper is organized as follows: In Section \ref{sec:soshierarchy} we 
prove that the Lasserre/sum-of-squares hierarchy of the regular $N$-gon requires exactly $\lceil N/4 \rceil$ iterations. Then in Section \ref{sec:construction} we show a construction of an equivariant psd lift of the regular $2^n$-gon of size $2n-1$. Finally in Section \ref{sec:lb} we show that any equivariant Hermitian psd lift of the regular $N$-gon must have size $\ln(N/2)$.

In the previous discussion we did not describe how a sum-of-squares certificate of the facet-defining inequalities of the regular $N$-gon leads to a psd lift. In the next section we make this connection more precise and we give the example of the regular hexagon as a simple illustration.

\subsection{Equivariant psd lifts from sum-of-squares certificates}
\label{sec:soslift}
Consider the facet-defining linear function $\ell(x,y) = \cos(\pi/N) - x$ of the regular $N$-gon.
In this section we see how sum-of-squares certificates of $\ell$ can be used to get an equivariant psd lift of the regular $N$-gon. The idea is actually quite general and applies to general polytopes, cf. \cite{gouveia2011lifts,lasserre2009convex}. We consider here the case of the regular $N$-gon for concreteness.

An important concept to describe the psd lift from the sum-of-squares certificate is the notion of \emph{moment map}. Let $V$ be a subspace of $\cF(N,\RR)$ (typically $V = \oplus_{k \in K} \TPol_k(N)$ for some $K\subseteq \{0,\dots,\lfloor N/2 \rfloor\}$) and let $d=\dim V$. Let $b_1,\dots,b_d$ be a basis of $V$ (a natural basis is the one given by cosine and sine functions $c_k,s_k$ for $k\in K$). Consider the following linear map:
\[
T_V : 
\begin{array}[t]{rl}
\S^d & \rightarrow \cF(N,\RR)\\
Q & \mapsto \sum_{i,j=1}^d Q_{ij} b_i b_j.
\end{array}
\]
Note that the map $T_V$ satisfies the following important property: a function $f \in \cF(N,\RR)$ can be written as a sum-of-squares of functions in $V$, if, and only if we can write $f = T_V(Q)$ where $Q$ is a positive semidefinite matrix, i.e.,
\[ \exists h_i \in V, \; f = \sum_{i} h_i^2 \quad \Longleftrightarrow \quad \exists Q \in \S^d_+, \; f = T_V(Q). \]
Now, define the \emph{moment map} of $V$, denoted $\cM_V$ to be the adjoint of $T_V$:
\[ \cM_V = (T_V)^* : \cF(N,\RR)^* \rightarrow \S^d, \]
where $\cF(N,\RR)^*$ is the dual space of $\cF(N,\RR)$ (i.e., the space of linear functions on $\cF(N,\RR)$), and where we identified $(\S^d)^*$ with $\S^d$ via the canonical inner product.
The definition of $\cM_V$ may look a bit abstract at first, but in fact $\cM_V$ can be quite easily computed in the basis of cosine and sine functions using trigonometric product formulae, as we shown in Example \ref{ex:hexagon} later.

We are ready to state the main result which shows how to obtain an equivariant psd lift of the regular $N$-gon from a sum-of-squares certificate for $\ell$:


\begin{thm}
\label{thm:momentlift}
Consider the facet-defining linear function $\ell(x,y) = \cos(\pi/N) - x$ of the regular $N$-gon. Assume that $\ell \in \cF(N,\RR)$ has a sum-of-squares representation in $\cF(N,\RR)$ of the form:
\begin{equation}
\label{eq:blocksos}
\ell = \sum_{i=1}^q \sum_j h_{ij}^2
\end{equation}
where for each $i=1,\dots,q$ and each $j$, the functions $h_{ij}$ are in a subspace $V_i$ of $\cF(N,\RR)$ that is invariant under the action of the dihedral group. Then the regular $N$-gon admits the following equivariant psd lift over the Cartesian product $\S^{d_1}_+ \times \dots \times \S^{d_q}_+$ where $d_i  =\dim V_i$ for each $i=1,\dots,q$:
%
\begin{equation}
\label{eq:liftmoment}
\conv(\Gon_N) = \Biggl\{ (z(c_1),z(s_1)) \; : \; z \in \cF(N,\RR)^* \text{ where } z(c_0) = 1, \cM_{V_i}(z) \in \S^{d_i}_+ \;\; \forall i=1,\dots,q \Bigr\}.
\end{equation}
In \eqref{eq:liftmoment}, $c_0,c_1,s_1$ are the cosine and sine functions in $\cF(N,\RR)$ defined in \eqref{eq:cksk}.
\end{thm}
\begin{proof}
  We omit the proof of this theorem since it is simply a restatement of the results in \cite{gouveia2010theta}. The proof that the lift is equivariant is given in \cite[Appendix A]{fawzi2014equivariant}.
\end{proof}
\begin{rem}
 Note that in the statement of the theorem we consider sum-of-squares certificates of $\ell$ of the form:
\[ \ell = \sum_{i=1}^q \sum_{j} h_{ij}^2 \]
where the functions $h_{ij}$ are in an invariant subspace $V_i$ of $\cF(N,\RR)$. The advantage of considering such a decomposition (over a coarser decomposition $\ell = \sum_{i} h_i^2$ where $h_i \in V$) is that when the subspaces $V_i$ are low-dimensional, this allows to obtain a psd lift using Cartesian products of small psd cones, rather than a psd lift with a single large psd matrix of size $\dim V_1 + \dots + \dim V_q$. This will be useful later in the construction of the psd lift for the regular $2^n$-gon in Section \ref{sec:construction}.
\end{rem}
\begin{example}[Regular hexagon]
\label{ex:hexagon}
We now illustrate the theorem above with the regular hexagon. Consider the facet-defining linear function $\ell = \cos(\pi/6) - x = \sqrt{3}/2 - x$ of the regular hexagon.
One can verify that we have the following sum-of-squares representation of $\ell$:


\begin{equation}
\label{eq:soshexagon}
\ell = \frac{\sqrt{3}}{4} \left(-1 + \frac{2}{\sqrt{3}} c_1\right)^2 + \frac{\sqrt{3}}{36} \left(-2s_1 + s_3\right)^2
\end{equation}
where the functions $c_1,s_1,s_3$ are as defined in \eqref{eq:cksk}. Theorem \ref{thm:momentlift} allows us to translate this sum-of-squares certificate into an equivariant psd lift. There are actually two ways to obtain the lift, depending on how we ``break'' the sum-of-squares certificate:
\begin{enumerate}
\item First one can see \eqref{eq:soshexagon} as a certificate of the type \eqref{eq:blocksos} where $q = 2$ and where $V_1 = \TPol_0 \oplus \TPol_1$ and $V_2 = \TPol_1 \oplus \TPol_3$. Observe that $\dim V_1 = \dim V_2 = 3$. To form the psd lift of the hexagon it thus remains to compute the moment maps $\cM_{V_1}$ and $\cM_{V_2}$ of $V_1$ and $V_2$. We show how this is done for $V_1$. First we choose the natural basis of $V_1$ given by cosine and sine functions: $c_0,c_1,s_1$. We then form all pairwise products of the basis elements of $V_1$ with each other, and express all these products in the basis of $\cF(N,\RR)$ formed with cosine and sine functions. This step mainly involves using trigonometric product formulae, and the fact that $c_{N-k} = -c_k$ and $s_{N-k} = s_k$. 
The matrix below shows the result of computing all the pairwise products of $c_0,c_1,s_1$:
\[
\bordermatrix[{[]}]{%
 & c_0 & c_1 & s_1\cr
c_0 & c_0 & c_1 & s_1\cr
c_1 & c_1 & (c_0 + c_2)/2 & s_2 / 2\cr
s_1 & s_1 & s_2 / 2 & (c_0 - c_2)/2
 }.
\]
The moment map $\cM_{V_1}$ simply maps each $z \in \cF(6,\RR)^*$ to the $3\times 3$ matrix above where each entry is replaced by its image by $z$, i.e.,
\[
\cM_{V_1}: z \in \cF(6,\RR)^* \mapsto 
\begin{bmatrix}
z(c_0) & z(c_1) & z(s_1)\\
z(c_1) & (z(c_0)+z(c_2))/2 & z(s_2)/2\\
z(s_1) & z(s_2)/2 & (z(c_0) - z(c_2))/2
\end{bmatrix}.
\]
Using the same procedure, one can show that the moment map of $V_2$ is given by (the basis of $V_2$ that we chose is $c_1,s_1,s_3$):
\[
\cM_{V_2}: z \in \cF(6,\RR)^* \mapsto 
\begin{bmatrix}
(z(c_0)+z(c_2))/2 & z(s_2)/2 & z(s_2)\\
z(s_2)/2 & (z(c_0) - z(c_2))/2 & z(c_2)\\
z(s_2) & z(c_2) & z(c_0)
\end{bmatrix}.
\]
Using Theorem \ref{thm:momentlift} we can thus write the following equivariant psd lift of the regular hexagon over $\S^3_+ \times \S^3_+$ (to make the expressions lighter, we let $u_i = z(c_i)$ and $v_i = z(s_i)$):
\[
\begin{aligned}
\conv(\Gon_6) = \Biggl\{ 
(u_1,v_1) \in \RR^2 : \exists u_2,v_2 \in \RR & \qquad 
\begin{bmatrix}
1 & u_1 & v_1\\
u_1 & (1+u_2)/2 & v_2 / 2\\
v_1 & v_2/2 & (1-u_2)/2
\end{bmatrix} \succeq 0\\
& \text{ and } 
\begin{bmatrix}
(1+u_2)/2 & v_2 / 2 & v_2\\
v_2/2 & (1-u_2)/2 & u_2\\
v_2 & u_2 & 1
\end{bmatrix} \succeq 0
\Biggr\}.
\end{aligned} 
\]
\item Using the sum-of-squares certificate \eqref{eq:soshexagon} we can actually obtain another lift of the regular hexagon using $\S^4_+$ (instead of $\S^3_+ \times \S^3_+$). To do this, we view \eqref{eq:soshexagon} as a certificate where the functions come from the single subspace $V = \TPol_0 \oplus \TPol_1\oplus \TPol_3$ (i.e., we take $q=1$ in the notation of \eqref{eq:blocksos}).  Observe that $\dim V = 4$ and that the moment map of $V$ is given by:
\[
\cM_{V}:z\in\cF(6,\RR)^* \mapsto
\begin{bmatrix}
z(c_0) & z(c_1) & z(s_1) & z(s_3)\\
z(c_1) & (z(c_0)+z(c_2))/2 & z(s_2)/2 & z(s_2)\\
z(s_1) & z(s_2)/2 & (z(c_0) - z(c_2))/2 & z(c_2)\\
z(s_3) & z(s_2) & z(c_2) & z(c_0)
\end{bmatrix}.
\]
Using Theorem \ref{thm:momentlift} we thus get the following equivariant psd lift of the regular hexagon of size 4:
\[
\conv(\Gon_6) = \left\{ (u_1,v_1) \in \RR^2 : \exists u_2,v_2,v_3 \in \RR \quad \begin{bmatrix}
1 & u_1 & v_1 & v_3\\
u_1 & (1+u_2)/2 & v_2 / 2 & v_2\\
v_1 & v_2/2 & (1-u_2)/2 & u_2\\
v_3 & v_2 & u_2 & 1
\end{bmatrix} \succeq 0 \right\}.
\]
\end{enumerate}
\end{example}

\section{Sum-of-squares hierarchy and nonnegative polynomial interpolation}
\label{sec:soshierarchy}

In this section we study the Lasserre/sum-of-squares hierarchy for the regular $N$-gon. Recall that the $k$'th level of the hierarchy is \emph{exact} if one can certify all the facet inequalities using sum-of-squares polynomials of degree smaller or equal than $k$. Since all the facets of the regular $N$-gon are equivalent up to rotation, it suffices to consider one facet, e.g., 
\[ \ell(x,y) = \cos(\pi/N) - x. \]
 The smallest $k$ for which we can write:
\[ \ell = \sum_{i} h_i^2 \]
where $\deg h_i \leq k$ (i.e., $h_i \in \TPol_0(N) \oplus \dots \oplus \TPol_k(N)$) is called the \emph{theta-rank} of the regular $N$-gon.  In this section we prove the following theorem:
\begin{thm}
The theta-rank of the regular $N$-gon is exactly $\lceil N/4 \rceil$.
\end{thm}

We first prove the lower bound which is apparently well-known though does not seem to be written explicitly anywhere. The argument we present below is due to G. Blekherman.
\begin{prop}
\label{prop:lbthetarank}
The sum-of-squares hierarchy of the regular $N$-gon requires at least $N/4$ iterations.
\end{prop}
\begin{proof}
Assume we can write:
\begin{equation}
\label{eq:ellSOS}
\ell(x,y) = SOS(x,y) + g(x,y)
\end{equation}
where $g(x,y)$ is a polynomial that vanishes on the vertices of the regular $N$-gon. We will show that $\deg SOS \geq N/2$. Since the trigonometric polynomial $g(\cos\theta,\sin\theta)$ has $N$ roots on the unit circle, it must have degree at least $N/2$ (a nonzero trigonometric polynomial of degree $d$ has at most $2d$ roots on the unit circle; also $g(\cos\theta,\sin\theta)$ is not identically zero because otherwise $\ell$ would be nonnegative on the whole unit circle which it is not because $\ell$ defines a facet of the $N$-gon). Thus we get that $\deg SOS = \deg (\ell - g) \geq N/2$ which is what we want.
\end{proof}

The rest of this section is mainly devoted to the proof that $\lceil N/4 \rceil$ iterations of the sum-of-squares hierarchy are sufficient for the regular $N$-gon. To do so we exploit the fact that the regular $N$-gon is a \emph{$k$-level} polytope where $k=\lceil N/2 \rceil$. In fact we develop new general results about the theta-rank of $k$-level polytopes.

\subsection{$k$-level polytopes and nonnegative polynomial interpolation}

In this section we study polytopes that are \emph{$k$-level}, and we see how this property implies an upper bound on the sum-of-squares hierarchy. The material presented in this section concerns general polytopes $P$, and is not restricted to the case of regular $N$-gons.

We first recall the definition of a $k$-level polytope (see e.g., \cite{gouveia2012convex}):
\begin{defn}
A polytope $P$ is called \emph{$k$-level} if every facet-defining linear function of $P$ takes at most $k$ different values on the vertices of the polytope.
\end{defn}
\begin{example}[Regular polygons]
\label{ex:regular_polygon_levels}
It is easy to verify that the regular $N$-gon is a $\lceil N/2 \rceil$-level polytope. Indeed the values taken by the facet-defining linear function $\ell(x,y) = \cos(\pi/N) - x$ on the $N$ vertices of the polytope are:
\[
0, \; x_{1,N} - x_{2,N}, \; \dots, \; x_{1,N} - x_{\lceil N/2 \rceil,N} \quad \text{ where } \quad 
x_{i,N} = \cos(\theta_i) = \cos\left(\frac{(2i-1)\pi}{N}\right). \]
By symmetry, the number of values taken by the other facet-defining linear functions is also $\lceil N/2 \rceil$, and thus the regular $N$-gon is $\lceil N/2 \rceil$-level.
\end{example}

It was shown in \cite[Theorem 11]{gouveia2012convex}, via a Lagrange interpolation argument, that if a polytope $P$ is $k$-level, then the $(k-1)$'st level of the sum-of-squares hierarchy is exact.
To prove this result, the idea is to look at a ``one-dimensional projection'' of the problem: Let $\ell(\mb{x}) \geq 0$ be a facet-defining linear inequality for $P$ and assume that $\ell(\mb{x})$ takes the $k$ values $0=a_0 < a_1 < \dots < a_{k-1}$ on the vertices of $P$. To get an upper bound on the sum-of-squares hierarchy for $P$, we need to express the function $\ell$ on the vertices of $P$ as a sum-of-squares. To do this, one can proceed as follows: let $p$ be a \emph{univariate} polynomial that satisfies $p(a_i) = a_i$ for all $i=0,\dots,k-1$ and that is \emph{globally nonnegative} on $\RR$. Since nonnegative univariate polynomials are sum-of-squares, this means that we can write $p = \sum_{i} h_i^2$  for some polynomials $h_i$. Then observe that for any vertex $\mb{x}$ of the polytope $P$ we have 
\[ \ell(\mb{x}) \overset{(*)}{=} p(\ell(\mb{x})) = \sum_{i} h_i(\ell(\mb{x}))^2, \]
where equality $(*)$  follows from the fact $\ell(x) \in \{a_0,\dots,a_{k-1}\}$ since $\mb{x}$ is a vertex of $P$.
This shows that $\ell$ coincides on the vertices of $P$ with a sum-of-squares polynomial of degree $d = \deg p$.
If one can find such a sum-of-squares certificate of degree $d$ for all the facet-defining linear functions of $P$, then this shows that the $d/2$-level of the sum-of-squares hierarchy is exact.

Note that there is a simple way to construct a polynomial $p$ that satisfies the required conditions, i.e., $p(a_i) = a_i$ for all $i=0,\dots,k-1$ and $p$ globally nonnegative: One can simply take a Lagrange interpolating polynomial $r$ of degree $k-1$ such that $r(a_i) = \sqrt{a_i}$ and then take $p(x) = r(x)^2$. The resulting polynomial $p$ has degree $2(k-1)$ and thus gives an upper bound of $k-1$ for the sum-of-squares hierarchy of $k$-level polytopes. This is the construction used in \cite[Theorem 11]{gouveia2012convex}.

It turns out however that one can sometimes find a polynomial $p$ that has smaller degree. This motivates the following definition:

\begin{defn}
Let $0=a_0 < a_1 < \dots < a_{k-1}$ be $k$ points on the real line. We say that the sequence $(a_0,\dots,a_{k-1})$ has \emph{nonnegative interpolation degree $d$} if there exists a globally nonnegative polynomial $p$ with $\deg p = d$ such that $p(a_i) = a_i$ for all $i=0,\dots,k-1$.
\end{defn}

The construction outlined above with Lagrange interpolating polynomials shows that any sequence of length $k$ has nonnegative interpolation degree at most $2(k-1)$.
Also note that the nonnegative interpolation degree of any sequence with $k$ elements must be at least $k$: indeed if $p$ has degree $\leq k-1$ and $p(a_i) = a_i$ for all $i=0,\dots,k-1$ then $p$ must be equal to the linear polynomial $x$, which is clearly not nonnegative.


The previous discussion concerning the sum-of-squares hierarchy for $k$-level polytopes can be summarized in the following proposition:

\begin{prop}
\label{prop:interpdegthetarank}
Let $P$ be a $k$-level polytope in $\RR^n$. Assume that for any facet-defining linear functional $\ell$ of $P$, the $k$ values taken by $\ell$ on the vertices of $P$ have nonnegative interpolation degree $d$. Then the $d/2$-iteration of the sum-of-squares hierarchy for $P$ is exact  (note that $d$ is necessarily even).
\end{prop}


In the rest of this section we study sequences of length $k$ that have nonnegative interpolation degree equal to $k$ (i.e., the minimum possible value). Let $k$ be an even integer and let $0=a_0 < a_1 < \dots < a_{k-1}$ be $k$ points on the positive real axis. The question that we thus consider is: does there exist a univariate polynomial $p$ such that:
\begin{equation}
\label{eq:pcond}
\left\{
\begin{array}{l}
\deg p = k\\[0.2cm]
p(a_i) = a_i \quad \forall i=0,\dots,k-1\\[0.2cm]
p(x) \geq 0 \quad \forall x \in \RR.
\end{array}\right.
\end{equation}

The next proposition gives a simple geometric characterization of the existence of a polynomial $p$ that satisfies \eqref{eq:pcond}:
\begin{prop}
\label{prop:interp}
Let $q(x) = \prod_{i=0}^{k-1} (x-a_i)$ be the monic polynomial of degree $k$ that vanishes at the $a_i$'s. Then the following are equivalent:\\
(i) There exists a polynomial $p$ that satisfies \eqref{eq:pcond};\\
(ii) The curve of $q(x)$ is above its tangent at $x=0$, i.e.:
\begin{equation}
\label{eq:qcond}
q(x) \geq q'(0) x \quad \forall x \in \RR.
\end{equation}
\end{prop}
\begin{proof}
Note that if a polynomial $p$ satisfies \eqref{eq:pcond} then it must be of the form:
\[
p(x) = \alpha q(x) + x
\]
where $\alpha$ is a scalar. Furthermore, since $p$ is nonnegative and $p(0) = 0$, then 0 must be a double root of $p$, i.e., $p'(0) = 0$. This means that we must have $\alpha q'(0) + 1 = 0$ which implies $\alpha = -1/q'(0)$. In other words, the unique $p$ which can satisfy conditions \eqref{eq:pcond} is the polynomial:
\[
p(x) = -\frac{q(x)}{q'(0)} + x.
\]
Observe that $p$ is, up to scaling, equal to the difference between $q(x)$ and its linear approximation at $x=0$:
\[ p(x) = -\frac{1}{q'(0)} ( q(x) - q'(0) x ). \]
Since $q'(0) < 0$ (since $k$ is even), we see that $p(x) \geq 0$ if and only if the curve of $q$ is above its linear approximation at $x=0$.
\end{proof}

\begin{example}[Equispaced and subadditive sequences]
To illustrate this result consider the sequence $a_i = i$ for $i=0,\dots,k-1$. Figure \ref{fig:plot_q_arithmetic} shows the plot of the polynomial $q(x) = \prod_{i=0}^{k-1} (x-a_i)$ for $k=6$ and its tangent at $x=0$. We see from the figure that the curve of $q$ is always above its linear approximation at $x=0$. This means, by Proposition \ref{prop:interp}, that the nonnegative interpolation degree of the sequence $0,1,\dots,5$ is $6$.
\begin{figure}[ht]
  \centering
  \includegraphics[width=7cm]{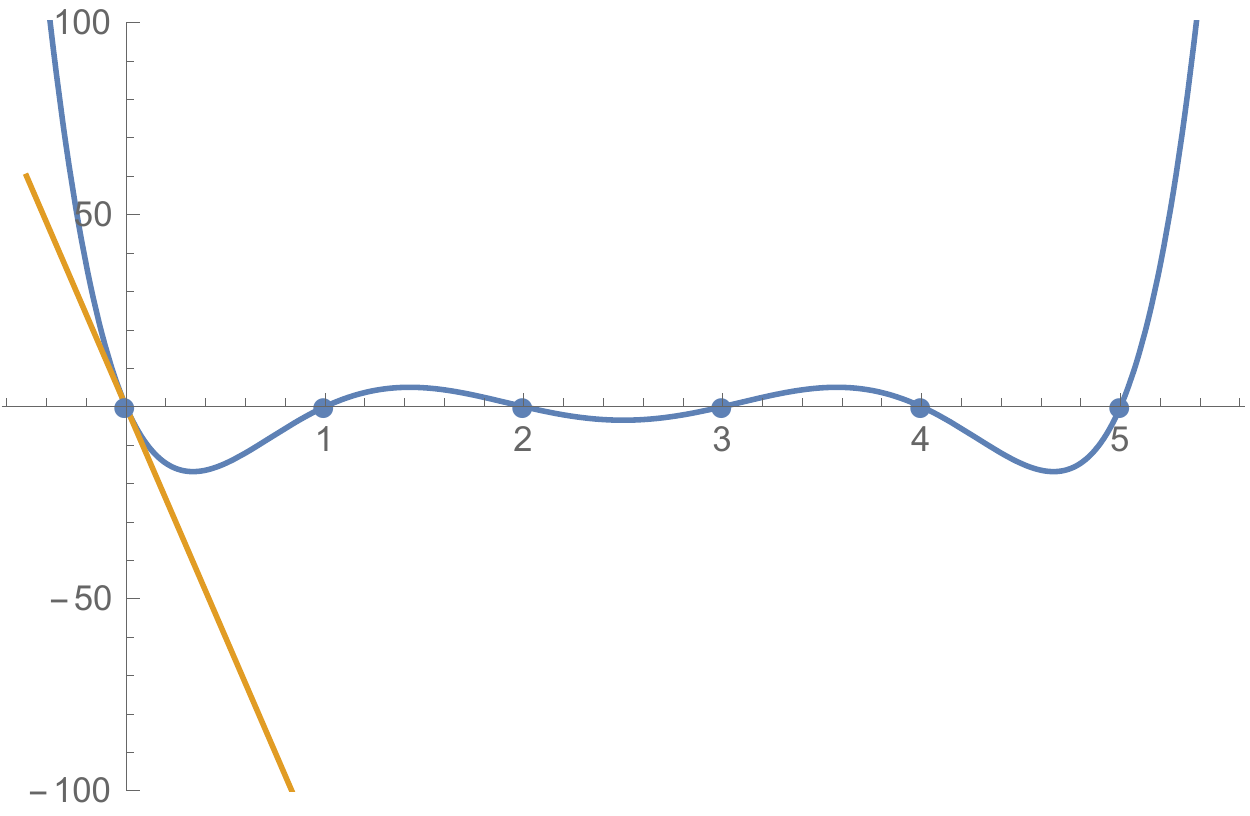}
  \caption{Plot of the polynomial $q(x) = \prod_{i=0}^5 (x - i)$ and its tangent at $x=0$. We see that the tangent is always below the curve of $q$. Thus from Proposition \ref{prop:interp} there is a polynomial $p$ of degree 6 that satisfies \eqref{eq:pcond} for the sequence $a_i = i$, $(i=0,\dots,5)$.}
  \label{fig:plot_q_arithmetic}
\end{figure}
One can can actually prove that the nonnegative interpolation degree of the sequence $a_i = i$ for $i=0,\dots,k-1$ is equal to $k$ for any $k$ even. In fact this is true even more generally for any sequence $a_0,\dots,a_{k-1}$ that is \emph{subadditive}, i.e., that satisfies $a_{i+j} \leq a_i + a_j$ for all $i,j$ such that $i+j \leq k-1$. This is the object of the next proposition:
\begin{prop}
\label{prop:subadditive}
Let $k$ be an even integer and assume that $0=a_0 < a_1 < \dots < a_{k-1}$ is a subadditive sequence, i.e., $a_{i+j} \leq a_i + a_j$ for all $i,j$ such that $i+j \leq k$. Then $(a_i)_{i=0,\dots,k-1}$ has nonnegative interpolation degree $k$; in other words there exists a globally nonnegative polynomial $p$ of degree $k$ such that $p(a_i) = a_i$ for all $i=0,\dots,k-1$.
\end{prop}
\begin{proof}[Sketch of proof]
Let $q(x) = \prod_{i=0}^{k-1} (x-a_i)$ and note that $q'(0) = -A$ where $A = \prod_{i=1}^{k-1} a_i$. To use Proposition \ref{prop:interp}, we need to show that the polynomial
\begin{equation}
\label{eq:diffsub}
 q(x) - q'(0) x = x \left[A + \prod_{i=1}^{k-1} (x-a_i)\right]
\end{equation}
is nonnegative for all $x \in \RR$. It is easy to see that the expression \eqref{eq:diffsub} is nonnegative if $x \leq 0$ or $x \geq a_{k-1}$. For $x \in (0,a_{k-1})$, let $y$ be the index in $\{0,1,\dots,k-2\}$ such that $a_{y} \leq x \leq a_{y+1}$. Then we have:
\[
\begin{aligned}
- \prod_{i=1}^{k-1} (x-a_i) \leq \left|\prod_{i=1}^{k-1} (x-a_i)\right| &= \prod_{i=1}^{y} (x-a_i) \cdot \prod_{i=y+1}^{k-1} (a_i-x) \\
&= [(x-a_1) (x-a_2) \dots (x-a_y)] \cdot [(a_{y+1}-x) \dots (a_{k-1}-x)]\\
&\overset{(a)}{\leq} [(a_{y+1} - a_1) (a_{y+1} - a_2) \dots (a_{y+1} - a_y)] \cdot [(a_{y+1} - a_y) \dots (a_{k-1} - a_y)]\\
&\overset{(b)}{\leq} (a_y a_{y-1} \dots a_1) (a_1 \dots a_{k-1-y})\\
&\overset{(c)}{\leq} (a_1 \dots a_y) (a_{y+1} \dots a_{k-1}) = A
\end{aligned}
\]
where in $(a)$ we used that $a_y \leq x \leq a_{y+1}$; in $(b)$ we used the subadditivity property of the sequence $(a_0,\dots,a_{k-1})$, and in $(c)$ we simply used the fact that $a_1 \leq a_{y+1}, a_2 \leq a_{y+2}, \dots, a_{k-1-y} \leq a_{k-1}$.
This shows that \eqref{eq:diffsub} is nonnegative for all $x \in (0,a_{k-1})$. Since \eqref{eq:diffsub} is also clearly nonnegative for all $x \leq 0$ and $x \geq a_{k-1}$, we can thus use Proposition \ref{prop:interp} to conclude the proof.
\end{proof}
\end{example}

\paragraph{Application for the parity polytope:} In \cite{fawzi2014equivariant} we considered the \emph{parity polytope} $\PAR_n$ defined as the convex hull of points in $\{-1,1\}^n$ that have an even number of $-1$'s:
\begin{equation}
 \label{eq:PARn}
 \PAR_n = \conv \left\{ x \in \{-1,1\}^n \; : \; \prod_{i=1}^n x_i = 1 \right\}.
\end{equation}
We showed in \cite[Proposition 3]{fawzi2014equivariant} that the sum-of-squares hierarchy for the parity polytope requires at least $n/4$ iterations. Using the interpolation argument given above for equispaced sequences, one can actually show that the theta-rank of the parity polytope is exactly $\lceil n/4 \rceil$. Indeed, it is not difficult to verify that the parity polytope is a $\lceil n/2 \rceil$-level polytope, and that the levels of each facet are \emph{equispaced}. Thus, by Proposition \ref{prop:interpdegthetarank} and since equispaced sequences of length $k$ have nonnegative interpolation degree $k$ (when $k$ is even) it follows that the theta-rank of the parity polytope is $\lceil n/4 \rceil$.
\begin{prop}
The theta-rank of the parity polytope $\PAR_n$ defined in \eqref{eq:PARn} is exactly $\lceil n/4 \rceil$.
\end{prop}

Any $2$-level polytope has theta-rank one (see, e.g.~\cite{gouveia2012convex}).
One way to see this is to note that any sequence $0=a_0<a_1$ of length
$2$ has nonnegative interpolation degree $2$. One can see this from the 
Lagrange interpolation argument given earlier, but perhaps more directly 
from Proposition~\ref{prop:interp}. In this 
case the polynomial $q(x) = (x-a_0)(x-a_1) = x(x-a_1)$ is convex and so 
its graph is certainly above its linear approximation at $x=0$.

Any sequence of length $4$ has nonnegative interpolation degree either 
$4$ or $6$ (since the Lagrange interpolation argument constructs 
a nonnegative interpolant of degree $6$). 
Furthermore, there is a simple characterization 
of those sequences of length $4$ that have nonnegative interpolation
degree $4$.
\begin{prop}
    A sequence $0=a_0<a_1<a_2<a_3$ of length $4$ has nonnegative
    interpolation degree $4$ if and only if 
\begin{equation}
\label{eq:disccone}
(a_1+a_2+a_3)^2 \leq 4(a_1a_2+a_1a_3+a_2a_3).
\end{equation}
\end{prop}
\begin{proof}
    We appeal to Proposition~\ref{prop:interp}. 
    In this case $q(x)-q'(0)x = x^2(x^2 - (a_1+a_2+a_3)x +
    (a_1a_2+a_1a_3+a_2a_3))$. This is nonnegative for all $x$ 
    if and only if the quadratic polynomial 
    $x^2 - (a_1+a_2+a_3)x + (a_1a_2+a_1a_3+a_2a_3)$ 
    is nonnegative for all $x$. This occurs precisely when the 
    discriminant is nonpositive, i.e.\
    \[ (a_1+a_2+a_3)^2 - 4(a_1a_2+a_1a_3+a_2a_3) \leq 0.\]
\end{proof}
Geometrically, the set of $(a_1,a_2,a_3)$ satisfying~\eqref{eq:disccone}
is the largest convex quadratic cone centred at $(1,1,1)$ that fits 
inside the nonnegative orthant. It is remarkable that
these sequences form a convex cone.

It would be interesting to understand, for general $k$, 
the set of sequences $0=a_0<a_1<\cdots < a_{k-1}$ of length $k$ with 
nonnegative interpolation degree $k$. For example,  
motivated by the construction of psd lifts of polytopes we pose the 
following problem.

\begin{question}
Give a simple (i.e., easy-to-check) sufficient condition for a sequence $0=a_0 < a_1 < \dots < a_{k-1}$ to have nonnegative interpolation degree $k$.
\end{question}

In this section we worked with ordered sequences $(a_i)_{i=0,\dots,k-1}$ that start at $a_0 = 0$ and we considered the problem of finding a nonnegative polynomial $p$ that takes the same values as the linear polynomial $x$ at the points $a_0,\dots,a_{k-1}$. For the regular polygon it will be convenient to work with shifted sequences, and with linear polynomials that have negative slope. We record the following result which we will use later, and which is an equivalent formulation of Proposition \ref{prop:interp}:
\begin{prop}
\label{prop:interp2}
Let $k$ be an even integer and let $a_0 > a_1 > \dots > a_{k-1}$ be $k$ points on the real axis. Let $l(x)$ be a \emph{decreasing} linear function with $l(a_i) \geq 0$ for $i=1,\dots,k-1$ and $l(a_0) = 0$. Let $q$ be the monic polynomial that vanishes on the $a_i$'s, $q(x) = \prod_{i=0}^{k-1} (x-a_i)$. 

If the curve of $q(x)$ is above its tangent at $x=a_0$ then there exists a polynomial $p$ of degree $k$ that is globally nonnegative and such that $p(a_i) = l(a_i)$ for all $i=0,\dots,k-1$.
\end{prop}

\subsection{Application to the theta-rank of regular polygons}

We now go back to the regular $N$-gon and use the results from the previous section to show that the theta-rank of the $N$-gon is $\lceil N/4 \rceil$. We focus on the facet inequality of the regular $N$-gon introduced earlier:
\begin{equation}
\label{eq:l1} 
\ell(x,y) = \cos(\pi/N) - x \geq 0.
\end{equation}
Our main result in this section is:
\begin{thm}
\label{thm:theta-rank-Ngon}
The linear function $\ell(x,y)$ agrees with a sum-of-squares polynomial of degree $2 \lceil N/4 \rceil$ on the vertices of the $N$-gon, i.e., there exist polynomials $h_i \in \RR[x]$ with $\deg h_i \leq \lceil N/4 \rceil$ such that
\[ \cos(\pi/N) - x = \sum_{i} h_i(x)^2 \quad \forall x \in \{\cos(\theta_1),\dots,\cos(\theta_N) \}. \]
\end{thm}

\begin{proof}
The proof of this theorem relies mainly on Proposition \ref{prop:interp2}. We consider first the case where $N$ is a multiple of 4; the other cases are similar but slightly more technical and are treated in Appendix \ref{app:theta-rank}. Thus assume $N = 4m$ where $m$ is an integer. Define 
\[ a_i = \cos(\theta_{i+1}) = \cos\left(\frac{(2i+1)\pi}{4m}\right) \quad i=0,\dots,2m-1 \]
and note that $a_0 > a_1 > \dots > a_{2m-1}$.
Let $l$ be the univariate linear polynomial $l(x) = a_0 - x$.
 Consider the polynomial $q$ which vanishes at the $a_i$'s:
\begin{equation}
\label{eq:q_chebyshev}
q(x) = \prod_{i=0}^{2m-1} (x - a_i) = \prod_{i=0}^{2m-1} \left(x - \cos\left(\frac{(2i+1)\pi}{4m}\right)\right).
\end{equation}
Note that, up to scaling, the polynomial $q$ is nothing but the Chebyshev polynomial of order $2m$. Indeed recall that the Chebyshev polynomial of degree $r$ has roots $\cos((2i+1)\pi / 2r)$, $i=0,\dots,r-1$ and coincides with the function $\cos(r\arccos(x))$ on $x \in [-1,1]$. Using this observation it is not difficult to show, using the properties of Chebyshev polynomials, that $q$ satisfies the condition of Proposition \ref{prop:interp2}, namely that the curve of $q$ lies above its linear approximation at $x=a_0$ (cf. Figure \ref{fig:plot_q_chebyshev} for a picture ($N=8$) and Lemma \ref{lem:chebyshev-tangent} in Appendix \ref{app:theta-rank} for a formal proof).

\begin{figure}[ht]
  \centering
  \includegraphics[width=7cm]{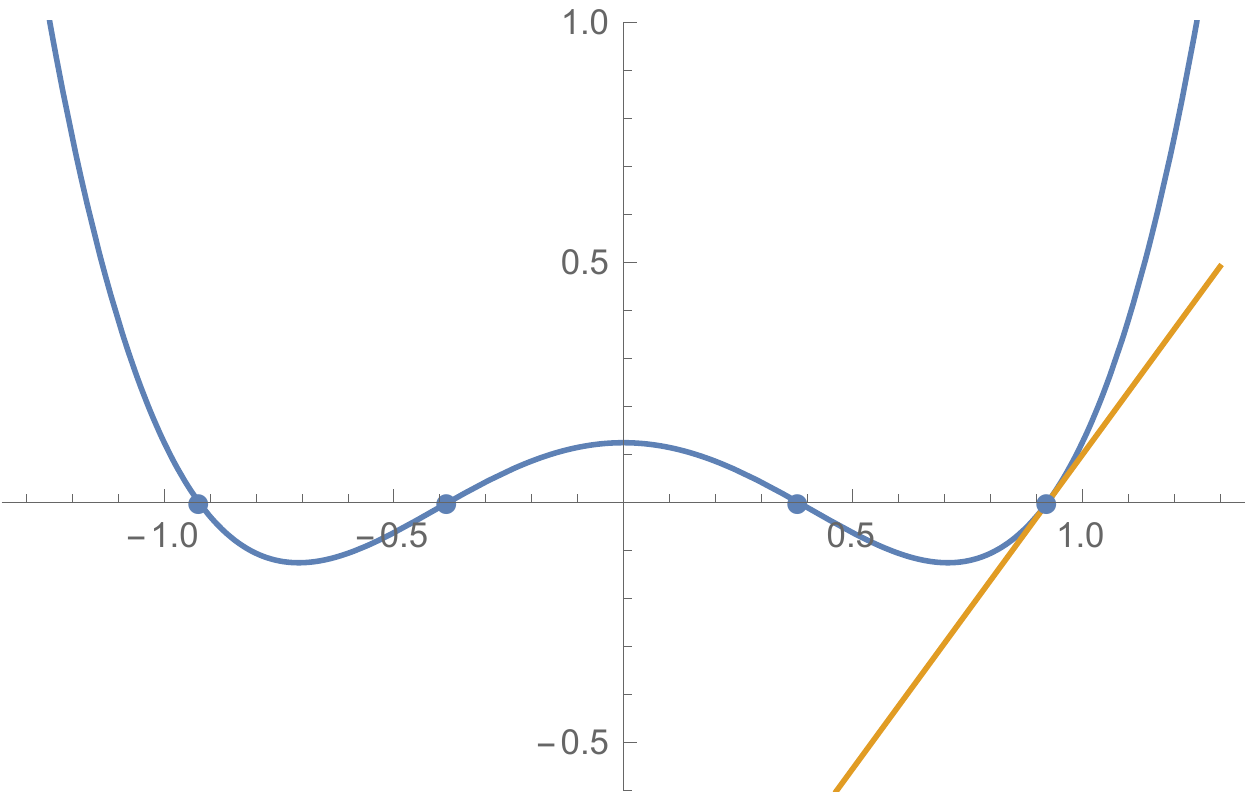}
\caption{Plot of the polynomial $q(x)$ of Equation \eqref{eq:q_chebyshev} for $2m = 4$ and its tangent at $x=a_0=\cos(\pi/8)$. We see that the tangent is always below the curve of $q$ (for a proof, cf. Lemma \ref{lem:chebyshev-tangent} in Appendix \ref{app:theta-rank}). Thus from Proposition \ref{prop:interp} there is a polynomial $p$ of degree 4 that is globally nonnegative and such that $p(a_i) = a_0 - a_i$ for all $i=0,\dots,2m-1$ where $a_i = \cos((2i+1)\pi/(4m))$.}
  \label{fig:plot_q_chebyshev}
\end{figure}

Thus from Proposition \ref{prop:interp2} it follows that there exists a nonnegative polynomial $p$ of degree $2m$ such that $p(a_i) = l(a_i) = a_0 - a_i$ for all $i=0,\dots,2m-1$.  Since nonnegative univariate polynomials are sum-of-squares we can write $p = \sum_{i} h_i^2$ where $h_i$ are polynomials of degree $\leq m$. Thus it follows that for any vertex $(x,y)$ of the regular $N$-gon, we can write:
\begin{equation}
 \label{eq:equalityell}
 \ell(x,y) = a_0 - x = l(x) \overset{(*)}{=} p(x) = \sum_{i} h_i(x)^2
\end{equation}
where in $(*)$ we used the fact that $x \in \{a_0\}_{i=0,\dots,2m-1}$ and that $p$ agrees with $l$ on the $a_i$'s. Thus this proves our claim in the case where $N$ is a multiple of four.

The proof when $N$ is not a multiple of four is slightly more technical for two reasons: the polynomial $q(x)$ is not necessarily a Chebyshev polynomial (though it is related), and the number of values that the facet $\ell(x,y)$ takes is not necessarily even. These cases are treated in detail in Appendix \ref{app:theta-rank}.
\end{proof}

\section{Construction}
\label{sec:construction}
In this section we construct two equivariant psd lifts of the 
regular $2^n$-gon. The first is a $(\S^3_+)^{n-1}$-lift, i.e., it expresses the regular $2^n$-gon using $n-1$ linear matrix inequalities of size $3\times 3$ each, whereas the second is a $\S^{2n-1}_+$-lift and uses a single linear matrix inequality of size $2n-1$. Both of our constructions 
arise from a sum of squares certificate of the non-negativity of 
$\ell(x,y) = \cos(\pi/2^{n})-x$ on the vertices of the regular $2^n$-gon 
(see Proposition~\ref{prop:facet1cert} to follow). Applying 
Theorem~\ref{thm:momentlift} in 
two different ways then gives the two different equivariant psd lifts 
of the regular $2^n$-gon.

We now establish the following sum of squares representation of the linear 
functional $\ell(x,y) = \cos(\pi/2^n)-x$ on the vertices of the regular
$2^n$-gon. Note that, in the space $\cF(N,\RR)$, this linear function can be expressed as $\ell = \cos(\pi/2^n) c_0 - c_1$.
\begin{prop}
    \label{prop:facet1cert}
    If $\ell = \cos(\pi/2^n)c_0 - c_1$ then, 
    in $\cF(2^n,\RR)$,
    \[ \ell = \sum_{k=0}^{n-2}
        \frac{\sin\left(\frac{\pi}{2^n}\right)}
        {2^k\sin\left(2^{k+1}\cdot\frac{\pi}{2^n}\right)}
        \left(\cos\left(2^k\cdot \frac{\pi}{2^n}\right)c_0 - c_{2^k}\right)^2.\]
\end{prop}
\begin{proof}
    To prove that $\ell$ has such a decomposition, it is sufficient to 
    establish that
    \begin{equation}
        \label{eq:sos}
        \frac{\cos\left(\frac{\pi}{2^n}\right) -
        \cos(\theta)}{\sin\left(\frac{\pi}{2^n}\right)}
        = \sum_{k=0}^{n-2}\frac{(\cos\left(2^k\cdot \frac{\pi}{2^{n}}\right)
    - \cos(2^k\theta))^2}{2^k\sin\left(2^{k+1}\cdot\frac{\pi}{2^n}\right)}
     - \frac{\cos(2^{n-1}\theta)}{2^{n-1}}\quad
     \text{for all $n\geq 1$ and all $\theta$}.
 \end{equation}
 This is enough to prove Proposition~\ref{prop:facet1cert}
 because $c_k$ is the restriction of $\cos(k\theta)$ to the angles
 $\theta_i = \frac{(2i-1)\pi}{2^n}$ for $i=1,2,\ldots,2^n$ corresponding 
 to the vertices of the regular $2^n$-gon, and $c_{2^{n-1}}=0$
 in $\cF(2^n,\RR)$. 
 
 We now establish the identity in~\eqref{eq:sos} by induction. 
 For the base case, observe that 
    $\frac{\cos(\pi/2)-\cos(\theta)}{\sin(\pi/2)} = -\cos(\theta)$
    which agrees with~\eqref{eq:sos} for $n=1$. 
    
    To take the induction step, we first prove the 
    following simple trigonometric identity that holds for all $N\geq 3$ and all $\theta$:
    \begin{equation}
        \label{eq:double}
        \frac{\cos\left(\frac{\pi}{N}\right) -
        \cos(\theta)}{\sin\left(\frac{\pi}{N}\right)} = 
        \frac{(\cos\left(\frac{\pi}{N}\right)-\cos(\theta))^2}
        {\sin\left(2\cdot\frac{\pi}{N}\right)} + 
        \frac{1}{2}\cdot\frac{\cos\left(2\cdot\frac{\pi}{N}\right) - \cos(2\theta)}{\sin\left(2\cdot\frac{\pi}{N}\right)}.
    \end{equation}
    To prove this identity, we start with the right-hand side, expand the square and use 
    the identity $\cos(2t) = 2\cos^2(t)-1$, then rewrite the denominator using $\sin(2t) =2\sin(t)\cos(t)$, i.e.,
    \begin{align*}
        \textup{RHS} = \frac{\left[\cos^2\left(\frac{\pi}{N}\right) - 
    2\cos\left(\frac{\pi}{N}\right)\cos(\theta)+\cos^2(\theta)\right] + 
\cos^2\left(\frac{\pi}{N}\right) -
\cos^2(\theta)}{\sin\left(2\cdot\frac{\pi}{N}\right)}
& = \frac{2\cos\left(\frac{\pi}{N}\right)(\cos\left(\frac{\pi}{N}\right) 
- \cos(\theta))}{\sin\left(2\cdot\frac{\pi}{N}\right)}\\
& = \frac{\cos\left(\frac{\pi}{N}\right) - \cos(\theta)}{\sin\left(\frac{\pi}{N}\right)}
\end{align*}
which is exactly the left-hand side.
   
    With~\eqref{eq:double} established, we return to our argument by induction. Assume that~\eqref{eq:sos} holds for some $n\geq 1$. 
    By first using~\eqref{eq:double} (with $N=2^{n+1}$), then applying the induction 
    hypothesis~\eqref{eq:sos} evaluated at $2\theta$ we have that:
\begin{align*}
    \frac{\cos\left(\frac{\pi}{2^{n+1}}\right) - \cos(\theta)}{\sin\left(\frac{\pi}{2^{n+1}}\right)} & = 
    \frac{(\cos\left(\frac{\pi}{2^{n+1}}\right) - \cos(\theta))^2}
    {\sin\left(\frac{\pi}{2^n}\right)} + 
    \frac{1}{2}\cdot \frac{\cos\left(\frac{\pi}{2^{n}}\right) - \cos(2\theta)}{\sin\left(\frac{\pi}{2^n}\right)}\\ 
    & = \frac{(\cos\left(2^0\cdot \frac{\pi}{2^{n+1}}\right) - \cos(2^0\cdot\theta))^2}
    {2^0\sin\left(2^{0+1}\cdot\frac{\pi}{2^{n+1}}\right)} + 
    \frac{1}{2}\left[\sum_{\ell=0}^{n-2}\frac{(\cos\left(2^{\ell}\cdot \frac{\pi}{2^{n}}\right)
        - \cos(2^\ell(2\theta)))^2}{2^\ell\sin\left(2^{\ell+1}\cdot\frac{\pi}{2^n}\right)}
    - \frac{\cos(2^{n-1}(2\theta))}{2^{n-1}}\right]\\
    & = \frac{(\cos\left(2^0\cdot \frac{\pi}{2^{n+1}}\right) - \cos(2^0\cdot\theta))^2}
    {2^0\sin\left(2^{0+1}\cdot\frac{\pi}{2^{n+1}}\right)} + 
    \left[\sum_{\ell=0}^{n-2}\frac{(\cos\left(2^{\ell+1}\cdot \frac{\pi}{2^{n+1}}\right)
        - \cos(2^{\ell+1}\theta))^2}{2^{\ell+1}\sin\left(2^{\ell+2}\cdot\frac{\pi}{2^{n+1}}\right)}
    - \frac{\cos(2^{n}\theta)}{2^{n}}\right]\\
     & = \sum_{k=0}^{n-1}\frac{(\cos\left(2^k\cdot 
 \frac{\pi}{2^{n+1}}\right)
    - \cos(2^k\theta))^2}{2^k\sin\left(2^{k+1}\cdot
    \frac{\pi}{2^{n+1}}\right)}
     - \frac{\cos(2^{n}\theta)}{2^{n}}
 \end{align*}
completing the proof.
 \end{proof}

In the context of Theorem~\ref{thm:momentlift} there are two natural ways to 
interpret the sum of squares decompostion of $\ell$ given in 
Proposotion~\ref{prop:facet1cert}. Both of these lead
to different equivariant lifts of the regular $2^n$-gon. In 
Sections~\ref{sec:prodlift} and~\ref{sec:smalllift} we describe these 
lifts.

\subsection{An equivariant $(\S_+^3)^{n-1}$-lift of the regular $2^{n}$-gon}
\label{sec:prodlift}
Let $V_i = \TPol_0(2^n)\oplus \TPol_{2^i}(2^n)$ for $i=0,2,\ldots,n-2$
and note that each $V_i$ has dimension $3$.
Then Proposition~\ref{prop:facet1cert} expresses 
$\ell = \cos\left(\frac{\pi}{2^n}\right)c_0 - c_1$ as
\[ \ell = \sum_{i=0}^{n-2} \sum_{j=1}^{1}h_{ij}^2\]
where each $h_{i1}\in V_i$. To apply Theorem~\ref{thm:momentlift} we 
need explicit expressions for the moment maps 
$\cM_{V_i}:\cF(2^n,\RR)^*\rightarrow \S^{3}$. 
For $i=0,1,\ldots,n-3$ we have
    \[ \cM_{V_i}:z\in \cF(2^n,\RR)^* \mapsto 
 \begin{bmatrix} z(c_0) & z(c_{2^i}) & z(s_{2^{i}})\\
            z(c_{2^i}) & (z(c_0)+z(c_{2^{i+1}}))/2 & z(s_{2^{i+1}})/2\\
            z(s_{2^i}) & z(s_{2^{i+1}})/2 & (z(c_0) - z(c_{2^{i+1}}))/2\end{bmatrix}.\]
    In the case of $V_{n-2}$, because 
    $c_{2^{n-2}}^2 = (c_0 - c_{2^{n-1}})/2=c_0/2$ in $\cF(2^n,\RR)$ 
    we have that
    \[ \cM_{V_{n-2}}:z\in \cF(2^n,\RR)^* \mapsto  \begin{bmatrix} z(c_0) & z(c_{2^{n-2}}) & z(s_{2^{n-2}})\\
            z(c_{2^{n-2}}) & z(c_0)/2 & z(s_{2^{n-1}})/2\\
            z(s_{2^{n-2}}) & z(s_{2^{n-1}})/2 & z(c_0)/2\end{bmatrix}.\]
    From these explicit expressions for the moment maps and 
    Theorem~\ref{thm:momentlift} we can obtain the equivariant lift
    of the regular $2^n$-gon in Theorem~\ref{thm:equivariantliftintro}, i.e.\ 
\begin{multline}
    \conv(\Gon_{2^n}) = \Biggl\{(x_0,y_0): \exists (x_i,y_i)_{i=1}^{n-2},y_{n-1},\quad 
        \begin{bmatrix} 1 & x_{k-1} & y_{k-1}\\x_{k-1} & \frac{1+x_k}{2} & \frac{y_k}{2}\\y_{k-1} &\frac{y_k}{2} & \frac{1-x_{k}}{2}\end{bmatrix} \psd 0
\quad
\text{for $k=1,2,\ldots,n-2$}\\
\text{and}\quad\begin{bmatrix} 1 & x_{n-2} & y_{n-2}\\x_{n-2} & \frac{1}{2} & \frac{y_{n-1}}{2}\\y_{n-2} & \frac{y_{n-1}}{2} & \frac{1}{2}\end{bmatrix} \psd 0.\Biggr\}
 \end{multline}
    Note that the variables $x_k$ and $y_k$ above correspond to the 
    variables $z(c_{2^k})$ and $z(s_{2^k})$ in Theorem~\ref{thm:momentlift}.

\subsection{An equivariant $\S_+^{2n-1}$-lift of the regular $2^n$-gon}
\label{sec:smalllift}

Let 
\[
V = \TPol_0(2^n)\oplus \TPol_{2^0}(2^n) \oplus \TPol_{2^1}(2^n)\oplus \cdots \oplus \TPol_{2^{n-2}}(2^n)
\]
 and note that $\dim(V) = 2n-1$. Then 
Proposition~\ref{prop:facet1cert} expresses $\ell = \cos\left(\frac{\pi}{2^n}\right)c_0 - c_1$ as
\[ \ell = \sum_{i=1}^{1}\sum_{j=0}^{n-2} h_{ij}^2\]
where $h_{1j}\in V$ for all $j$. Thus Theorem~\ref{thm:momentlift} shows that the regular $2^n$-gon admits the following equivariant psd lift of size $2n-1$:
\begin{equation}
\conv(\Gon_{2^n}) = \Biggl\{(z(c_1),z(s_1)) : z \in \cF(2^n,\RR)^*, z(c_0) = 1 \text{ and } \cM_V(z) \succeq 0 \Biggr\}
\end{equation}
where $\cM_{V}:\cF(2^n,\RR^n)^*\rightarrow \S^{2n-1}$ is the moment map for $V$.
One can compute the moment map of $V$ using trigonometric identities though it may be complicated to write explicitly for large values of $n$.
For illustration we computed the moment map of $V$ for the case $n = 4$ (the 16-gon) and we get that the regular 16-gon is the set of $(u_1,v_1)$ for which the following $7\times 7$ matrix is positive semidefinite:
\[
    \begin{bmatrix}
    2      &  2u_{1}     & 2v_{1}        &2u_{2}         & 2v_{2}         &  2u_{4}       & 2v_{4}       \\
    2u_{1} &  1+u_{2}    &  v_{2}        &  u_{1}+u_{3}  &  v_{1}+v_{3}   &  u_{3}+u_{5}  &   v_{3}+v_{5}\\
    2v_{1} &  v_{2}      &  1-u_{2}      &  -v_{1}+v_{3} &  u_{1}-u_{3}   &  -v_{3}+v_{5} &   u_{3}-u_{5}\\
    2u_{2} &  u_{1}+u_{3}&   -v_{1}+v_{3}&    1+u_{4}    &   v_{4}        &  u_{2}+u_{6}  &   v_{2}+v_{6}\\
    2v_{2} &  v_{1}+v_{3}&   u_{1}-u_{3} &    v_{4}      &   1-u_{4}      &  -v_{2}+v_{6} &   u_{2}-u_{6}\\
    2u_{4} &  u_{3}+u_{5}&   -v_{3}+v_{5}&    u_{2}+u_{6}&    -v_{2}+v_{6}&    1          &   v_{8}      \\
    2v_{4} &  v_{3}+v_{5}&   u_{3}-u_{5} &    v_{2}+v_{6}&    u_{2}-u_{6} &    v_{8}      &    1          
    \end{bmatrix}.
\]

\section{Lower bound on equivariant psd lifts of regular polygons}
\label{sec:lb}

In this section we are interested in obtaining lower bounds on equivariant psd lifts of the regular $N$-gon.  For convenience we will consider \emph{Hermitian} psd lifts, which are psd lifts defined with the cone of positive semidefinite Hermitian matrices, denoted $\HH^d_+$, instead of the cone $\S^d_+$ of psd real symmetric matrices. Clearly any lower bound for psd lifts over $\HH^d_+$ is also a lower bound for psd lifts over $\S^d_+$. The definition of equivariance for Hermitian psd lifts is the same as Definition \ref{def:equivariance} except that the cone $\S^d_+$ is replaced by $\HH^d_+$ and transposes are replaced by Hermitian conjugates (the Hermitian conjugate of a matrix $A$ is denoted $A^*$).

Let $\Rot_N$ be the subgroup of rotations of the dihedral group for the $N$-gon. Note that $\Rot_N \cong \ZZ_N$.
The main result of this section is the following:

\begin{thm}
\label{thm:lbhermitian}
Any Hermitian psd lift of the regular $N$-gon that is equivariant with respect to $\Rot_N$ has size at least $\ln(N/2)$.
\end{thm}

The proof of this theorem first relies on the structure theorem from \cite{fawzi2014equivariant} which says that any equivariant psd lift of size $d$ of the regular $N$-gon gives a sum-of-squares certificate of facet inequalities using trigonometric polynomials that are \emph{$d$-sparse} (we recall the precise statement of this as Theorem \ref{thm:structure}, to follow). The main part of this section is then dedicated to showing that any such sum-of-squares certificate requires $d$ to be at least $\ln( N/2 )$.

We now introduce some notations which will be used throughout the section.

\paragraph{Notations and terminology} We denote by $\cF(N,\CC)$ the space of complex-valued functions on the vertices of the $N$-gon.
For $k \in \ZZ$ consider the element $e_k \in \cF(N,\CC)$ defined by (as before, we identify functions on the vertices of the regular $N$-gon as functions on the set of angles $\{\theta_1,\dots,\theta_N\}$) \footnote{The constant $e^{-ik\pi/N}$ in the definition of $e_k$ makes calculations more convenient. For example with this definition we have $e_{k+N} = e_k$, whereas otherwise there is a minus sign: $e^{i(k+N)\theta} = -e^{ik\theta}$ for $\theta \in \{\theta_1,\dots,\theta_N\}$ since the $\theta_i$ are odd multiples of $\pi/N$.}:
\[ e_k(\theta) = e^{-ik\pi/N} e^{ik\theta} \quad \forall \theta \in \{\theta_1,\dots,\theta_N\}. \]
Note that for any $k \in \ZZ$ we have $e_{k+N} = e_k$, thus the element $e_k$ only depends on the residue class of $k$ modulo $N$. It is thus natural to index the elements $e_k$ with $k \in \ZZ_N$ instead of $k \in \ZZ$. Also note that we have $e_k^* = e_{-k}$ where $^*$ denotes complex conjugation, and for $k,k' \in \Freq{N}$ we have $e_k e_{k'} = e_{k+k'}$ where $e_k e_{k'}$ denotes pointwise multiplication of the functions $e_k$ and $e_{k'}$. 


The space $\cF(N,\CC)$ decomposes into a direct sum of one-dimensional spaces spanned by the $e_k$'s:
\[ \cF(N,\CC) = \bigoplus_{k \in \Freq{N}} \CC e_k. \]
Note that each $\CC e_k$ is an invariant subspace of $\cF(N,\CC)$ under the action of $\Rot_N$. If $h \in \cF(N,\CC)$, then the decomposition of $h$ in the basis $(e_k)_{k \in \Freq{N}}$ corresponds to the discrete Fourier transform of $h$. For this reason, we will often refer to the index $k$ in $e_k$ as a ``frequency''.

The following definition will be useful later:
\begin{defn}
Given $h \in \cF(N,\CC)$ and $K \subseteq \Freq{N}$, we say that $h$ is \emph{supported on} $K$ and we write $\supp h \subseteq K$ if $h$ is a linear combination of the elements $\{e_k : k \in K\}$.
\end{defn}

Recall the facet inequality of the regular $N$-gon given by $\cos(\pi/N) - x \geq 0$. The facet linear functional $\cos(\pi/N) - x$ can be expressed in the Fourier basis $(e_k)_{k \in \Freq{N}}$ as:
\begin{equation}
 \label{eq:elldef}
 \ell = \cos(\pi/N) e_0 - \frac{1}{2}(e^{i\pi/N} e_1 + e^{-i\pi/N} e_{-1}).
\end{equation}
We are interested in certificates of nonnegativity of $\ell$ using sums of hermitian squares of the form:
\begin{equation}
 \label{eq:ellsos}
 \ell = \sum_{i} |h_i|^2
\end{equation}
where $h_i \in \cF(N,\CC)$. More precisely we are interested in certificates where the functions $h_i$ are supported on a ``small'' set $K \subseteq \Freq{N}$. 
For convenience, we introduce the following definition of an \emph{sos-valid} set $K$:
\begin{defn}
A set $K\subseteq \Freq{N}$ is called \emph{sos-valid} if \eqref{eq:ellsos} holds where $\supp h_i \subseteq K$ for all $i$.
\end{defn}

\subsection{Structure theorem for regular $N$-gons}

In \cite{fawzi2014equivariant} we studied certain class of polytopes known as \emph{regular orbitopes} and we established a connection between equivariant psd lifts for such polytopes and sum-of-squares certificates of facet-defining inequalities.  The regular $N$-gon can be shown to be a $\Rot_N$-regular orbitope and so one can apply the results from \cite{fawzi2014equivariant} to characterize equivariant psd lifts of the regular $N$-gon. 
We summarize this characterization in the following theorem and we include a proof for completeness:

\begin{thm}
\label{thm:structure}
Assume that the regular $N$-gon has a Hermitian psd lift of size $d$ that is equivariant with respect to $\Rot_N$. Then there exists a set $K \subseteq \Freq{N}$ with $|K|\leq d$ that is sos-valid, i.e., there exist functions $h_i \in \cF(N,\CC)$ with $\supp h_i \subseteq K$ such that:
\[ \ell = \sum_{i} |h_i|^2 \]
where $\ell$ is the facet-defining linear functional of the regular $N$-gon defined in \eqref{eq:elldef}.
\end{thm}
\begin{proof}
Let $\Theta_N := \{\theta_1,\dots,\theta_N\}$ where $\theta_i = (2i-1)\pi/N$ be the angles of the vertices of the regular $N$-gon. 
We use the \emph{factorization theorem} for equivariant psd lifts (cf. \cite[Theorem A]{fawzi2014equivariant}), which is the analogue of Yannakakis' theorem \cite{yannakakis1991expressing} for equivariant psd lifts. Since we have a Hermitian psd lift of the regular $N$-gon of size $d$, the factorization theorem says that there exists a map $A:\Theta_N \rightarrow \HH^d_+$ and $B \in \HH^d_+$ such that:
\[ \ell(\theta) = \langle A(\theta), B \rangle \quad \forall \theta \in \Theta_N. \]
Note that we identified vertices of the regular $N$-gon with the set $\Theta_N$.
Furthermore, since the lift is equivariant, the map $A$ satisfies the following equivariance relation:
\[ A(r\cdot \theta) = \rho(r) A(\theta) \rho(r)^* \quad \forall r \in \Rot_N, \; \forall \theta \in \Theta_N \]
where $\rho : \Rot_N \rightarrow GL_d(\CC)$ is a group homomorphism and where $r\cdot \theta$ denotes the natural action of $\Rot_N$ on $\Theta_N$.
Since $\Rot_N \cong \ZZ_N$, we will identify $\Rot_N$ with $\ZZ_N$ in the rest of the proof. 
Note that $\rho$ is nothing but a $d$-dimensional linear representation of $\ZZ_N$. Since the irreducible representations of $\ZZ_N$ are all one-dimensional, there is a change-of-basis matrix $T \in GL_d(\CC)$ so that $\rho(r)$ is diagonal, i.e., we can write:
\[ \rho(r) = T \diag(t(r)) T^{-1} \quad \forall r \in \ZZ_N, \]
where $t=(t_1,\dots,t_d):\ZZ_N\rightarrow (\CC^*)^{d}$. Note that for each $j=1,\dots,d$, the map $t_j:\ZZ_N \rightarrow \CC^*$ is a group homomorphism and thus takes the form \begin{equation}
\label{eq:tj}
 t_j(r) = e^{2ik_j r \pi/N} \quad \forall r \in \ZZ_N
\end{equation}
where $k_j \in \Freq{N}$. Let
\begin{equation} K = \{k_1,\dots,k_d\} \subseteq \Freq{N}
\end{equation} and note that $|K| \leq d$. We will now show that $K$ is sos-valid, i.e., that $\ell$ has a sum-of-squares representation using functions supported on $K$.

Let $\theta_1 := \pi/N$ and observe that, by the equivariance relation on $A$, we have: $A(r\cdot \theta_1) = \rho(r) A(\theta_1) \rho(r)^*$ for any $r \in \ZZ_N$. Thus we have, for any $r \in \ZZ_N$:
\[ 
\begin{aligned}
\ell(r\cdot \theta_1 ) &= \Tr[\rho(r) A(\theta_1) \rho(r)^* B^*] \\
&= \Tr[T \diag(t(r)) T^{-1} A(\theta_1) T^{-*} \diag(t(r))^* T^* B^*] \\
&\overset{(a)}{=} \Tr[\diag(t(r)) A' \diag(t(r))^* B'^*] \\
&\overset{(b)}{=} t(r)^* (A' \hadprod B') t(r)
\end{aligned}
 \]
where in $(a)$ we used $A' = T^{-1} A(\theta_1) T^{-*}$ and $B' = T^* B T$ and in $(b)$ we denoted by $A' \hadprod B'$ the Hadamard (componentwise) product of $A'$ and $B'$. Since $A',B'$ are positive semidefinite, $A'\hadprod B'$ is positive semidefinite too (by the Schur product theorem) and thus we can write
\[ A'\hadprod B' = \sum_{i} v_i v_i^* \]
where $v_i \in \CC^d$. Thus we finally get that:
\begin{equation}
 \label{eq:sosell1}
 \ell(r\cdot \theta_1) = \sum_{i} |v_i^* t(r)|^2 = \sum_{i} |h_i(r)|^2 \quad \forall r \in \ZZ_N
\end{equation}
where $h_i := v_i^* t:\ZZ_N\rightarrow \CC$ are linear combinations of the $t_j$'s given in \eqref{eq:tj}. Since $r \cdot \theta_1$ ranges over the set $\Theta_N$ as $r$ ranges over $\ZZ_N$, Equation \eqref{eq:sosell1} can be rewritten as:
\begin{equation}
\ell(\theta) = \sum_{i} |\hat{h_i}(\theta)|^2 \quad \forall \theta \in \Theta_N
\end{equation}
where for $\theta \in \Theta_N$ we let $\hat{h_i}(\theta) = h_i(r)$ with $r$ being the unique element in $\ZZ_N$ such that $r\cdot \theta_1 = \theta$. Since the $h_i$ are linear combinations of the pure frequencies $t_j$ given in \eqref{eq:tj}, it is easy to see that the functions $\hat{h_i} \in \cF(N,\CC)$ are supported on $K$. Thus this completes the proof.
\end{proof}

Theorem \ref{thm:structure} thus reduces the problem of studying equivariant psd lifts of the regular $N$-gon to the problem of studying sets $K \subseteq \Freq{N}$ that are \emph{sos-valid}. The remaining part of this section is thus devoted to the study of such sets, and in particular to obtaining a lower bound on the size of sos-valid sets.

\subsection{Necessary conditions for a set to be sos-valid}


In this section we give a necessary condition on the ``geometry'' of a set $K$ to be sos-valid. Before stating the theorem, we make some observations and definitions:

First, observe that if $K$ is a set that is sos-valid, then any \emph{translation} $K'=K+t$ of $K$ is also sos-valid, where $t \in \ZZ_N$. This is because if $\ell = \sum_{i} |h_i|^2$ where $\supp h_i \subseteq K$, then we have $\ell = \sum_i |h'_i|^2$ where $h'_i = e_t h_i$ are supported on $K'$.

Second, it is useful to think of $\Freq{N}$ as the nodes of a cycle graph of length $N$, and of a set of frequencies $K \subseteq \Freq{N}$ as a subset of the nodes of this graph. For example Figure \ref{fig:freqcyclegraph} shows a set $K$ with $|K| = 7$ for the $N=12$-gon (the elements of $K$ are the black dots).  Note that since the property of being sos-valid is invariant under translation, the cycle graph need not be labeled. The only information that matters are the relative distances of the elements of $K$ with respect to each other.
\newdimen\R
\R=1.5cm
\newdimen\Rs
\Rs=0.1cm
\def\Nnodes{12}
\def\KK{0,1,3,4,6,7,8}
\begin{figure}[ht]
\centering
\begin{tikzpicture}
    \draw[black] circle (\R);
    \foreach \i in {1,...,\Nnodes} {
      \filldraw[fill=white,draw=black] (90-\i*360/\Nnodes:\R) circle (\Rs);
    }
    \foreach \i in \KK {
      \fill[black] (90-\i*360/\Nnodes:\R) circle (\Rs);
    }
\end{tikzpicture}
\caption{A set of frequencies $K$ for the regular $12$-gon.}
\label{fig:freqcyclegraph}
\end{figure}
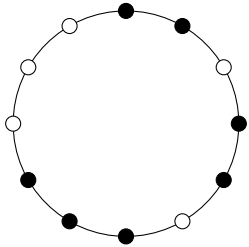

We endow $\Freq{N}$  with the natural distance $d$ on the cycle graph. The distance between two frequencies $k,k' \in \Freq{N}$ is denoted by $d(k,k')$; also if $C,C'$ are two subsets of $\Freq{N}$ we let
\[ d(C,C') = \min_{k \in C, k' \in C'} d(k,k'). \]
If $x \in \Freq{N}$ and $r$ is a positive integer, we can define the ball $B(x,r)$ centered at $x$ and with radius $r$ to be the set $B(x,r) := \{y \in \Freq{N}:d(x,y) \leq r\}$. We also let $[x,x+r]$ be the interval $\{x,x+1,\dots,x+r\} \subseteq \Freq{N}$. Note that the ball centered at $x$ of radius $r$ is simply the interval $[x-r,x+r]$.

In Section \ref{sec:soshierarchy}, Proposition \ref{prop:lbthetarank} we showed that the linear functional $\ell$ of the regular $N$-gon does not admit any sum-of-squares certificate with polynomials of degree smaller than $N/4$. One can state this result in a different way as follows: If $K$ is a set of frequencies that is included in a ball of radius smaller than $N/4$, then $K$ is not sos-valid.
The goal of this section is to extend this result and give a more general necessary condition for a set $K$ to be sos-valid in terms of its geometry. 

To state the main theorem, it will be more convenient to work with diameters instead of radii of balls (mainly to avoid the issue of dividing by two). We introduce the notion of \emph{in-diameter} of a set $K$ which is essentially twice the radius of the smallest ball containing $K$. More formally we have:

\begin{defn}
Let $K \subseteq \Freq{N}$. We define the \emph{in-diameter} of $K$, denoted $\indiam(K)$ to be the smallest positive integer $r$ such that $K$ is included in an interval of the form $[x,x+r]$ where $x \in \Freq{N}$.
\end{defn}

\begin{rem}
Note that the in-diameter of a set $K$ is in general different from the usual notion of \emph{diameter} (largest distance between two elements in $K$). 
Note for example that $\indiam(\Freq{N}) = N$ whereas the diameter of $\Freq{N}$ is equal to $\lfloor N/2 \rfloor$. 
\end{rem}

We are now ready to state the main result of this section:
%
%
\begin{thm}
\label{thm:clustering}
Let $N$ be an integer and let $K \subseteq \Freq{N}$ be a set of frequencies. Assume that $K$ can be decomposed into disjoint clusters $(C_{\alpha})_{\alpha \in A}$:
\[ K = \bigcup_{\alpha \in A} C_{\alpha}, \]
such that the following holds for some $1 \leq \gamma < N/2$:
\begin{enumerate}
\item[(i)] For any $\alpha \in A$, $C_{\alpha}$ has in-diameter $\leq \gamma$.
\item[(ii)] For any $\alpha \neq \alpha'$, $d(C_{\alpha},C_{\alpha'}) > \gamma$.
\end{enumerate}
Then the set $K$ is \emph{not} sos-valid (i.e., it is \emph{not} possible to write the linear function $\ell$ as a sum of squares of functions supported on $K$).
\end{thm}
\begin{proof}

To prove this theorem, we will construct a linear functional $\cL$ on $\cF(N,\CC)$ such that:
\begin{itemize}
\item[(a)] $\cL(\ell) < 0$, and;
\item[(b)] for any $h$ supported on $K$ we have $\cL(|h|^2) \geq 0$.
\end{itemize}
Clearly this will show that we cannot have $\ell = \sum_{i} |h_i|^2$ where $\supp h_i \subseteq K$.

We first introduce a piece of notation that will be needed for the definition of $\cL$: Given $k \in \Freq{N}$, we let $k\bmod N$ be the unique element in 
\[ \Bigl\{-\lceil N/2 \rceil+1,\dots,\lfloor N/2 \rfloor\Bigr\} \]
that is equal to $k$ modulo $N$. The main property that will be used about this operation is the following, which can be verified easily: If $k,k' \in [0,\gamma]$ where $\gamma< N/2$ then:
\begin{equation}
\label{eq:identitymod}
 (k'-k) \bmod N = (k' \bmod N) - (k \bmod N).
\end{equation}
We now define the linear functional $\cL:\cF(N,\CC)\rightarrow \CC$ as follows, for all $k \in \Freq{N}$:
\begin{equation}
\label{eq:defL}
\cL(e_k) = \begin{cases} e^{-\frac{i\pi}{N}(k \bmod N)} & \text{ if } d(0,k) \leq \gamma\\
                            0 & \text{ else.} \end{cases}
\end{equation}
The map $\cL$ defined here can actually be interpreted in terms of a point evaluation ``outside'' the regular $N$-gon, cf. Remark \ref{rem:mapL} for more details.
We now prove that $\cL$ satisfies properties (a) and (b) above.
\begin{itemize}
\item[(a)] It is easy to see that $\cL(\ell)  < 0$. Indeed since $\gamma \geq 1$ we have $\cL(e_1) = e^{-i\pi/N}$ and $\cL(e_{-1}) = e^{i\pi/N}$ which implies that:
\[ \cL(\ell) = \cL\Bigl(\cos(\pi/N) e_0 - (e^{i\pi/N} e_1 + e^{-i\pi/N}e_{-1})/2\Bigr) = \cos(\pi/N) - 1 < 0. \]

\item[(b)] We now show that if $h$ is function supported on $K$, then $\cL(|h|^2) \geq 0$.
Since $K = \cup_{\alpha \in A} C_{\alpha}$, we can write
\[ h = \sum_{k \in K} h_k e_k = \sum_{\alpha \in A} \sum_{k \in C_{\alpha}} h_{k} e_k. \]
Thus
\begin{equation}
\label{eq:decomph2}
|h|^2 = h^* h = \underbrace{\sum_{\alpha \in A} \left| \sum_{k \in C_{\alpha}} h_k e_k \right|^2}_{P} + \underbrace{\sum_{\alpha \neq \alpha'} \sum_{k \in C_{\alpha}, k' \in C_{\alpha'}} h_k^* h_{k'} e_k^* e_{k'}}_{Q}.
\end{equation}
Let $P$ and $Q$ be the first and second terms in the equation above. We will show that $\cL(Q) = 0$ and that $\cL(P) \geq 0$.
Observe that if $k \in C_{\alpha}$ and $k' \in C_{\alpha'}$ where $\alpha\neq \alpha'$ then we have:
\[ \cL(e_k^* e_{k'}) = \cL(e_{k' - k}) = 0 \]
where the last equality follows since $d(k'-k,0) = d(k',k) > \gamma$ (cf. assumption (ii) on the clustering). Thus this shows that $\cL(Q) = 0$.

We will now show that $\cL(P) \geq 0$, by showing that for any $\alpha \in A$ we have 
\[ \cL\left(\left| \sum_{k \in C_{\alpha}} h_k e_k \right|^2\right) \geq 0. \]
Let $\alpha \in A$. By assumption (i) on the clustering, we know that the in-diameter of $C_{\alpha}$ is $\leq \gamma$, i.e., that $C_{\alpha}$ is included in an interval $[x,x+\gamma]$. Note that since
\[ \left| \sum_{k \in C_{\alpha}} h_k e_k \right|^2 = \left| e_{-x} \sum_{k \in C_{\alpha}} h_k e_k \right|^2 = \left| \sum_{k \in C_{\alpha}} h_k e_{k-x} \right|^2 \]
we can assume without loss of generality that $x=0$.
Now since $C_{\alpha} \subseteq [0,\gamma]$, we have from \eqref{eq:identitymod} that for any $k,k' \in C_{\alpha}$:
\begin{equation}
 \label{eq:identitymod2}
 (k'-k) \bmod N = (k' \bmod N) - (k \bmod N)
\end{equation}
Using this we have:
\[ 
\label{eq:eqsii}
\begin{aligned}
\cL\left(\left|\sum_{k \in C_{\alpha}} h_k e_k\right|^2\right) 
= \sum_{k,k' \in C_{\alpha}} h_k^* h_{k'} \cL(e_{k'-k}) 
&\overset{(a)}{=} \sum_{k,k' \in C_{\alpha}} h_k^* h_{k'} e^{-\frac{i\pi}{N} ( (k'-k) \bmod N ) }\\
&\overset{(b)}{=} \sum_{k,k' \in C_{\alpha}} h_k^* h_{k'} e^{-\frac{i\pi}{N} ( k' \bmod N )} e^{\frac{i\pi}{N} (k \bmod N )} \\
&= \left|\sum_{k \in C_{\alpha}} h_k e^{-\frac{i\pi}{N} ( k \bmod N )} \right|^2 \geq 0
\end{aligned}
 \]
where in $(a)$ we used the fact that $d(0,k'-k) = d(k',k) \leq \gamma$ and in $(b)$ we used identity \eqref{eq:identitymod2}.
Thus this shows that $\cL(|h|^2) \geq 0$ for all $h$ supported on $C_{\alpha}$, which implies that $\cL(P) \geq 0$ (since $P = \sum_{\alpha \in A} \left| \sum_{k \in C_{\alpha}} h_k e_k \right|^2$) which is what we wanted.
\end{itemize}

\end{proof}

\begin{rem}
To illustrate the previous theorem consider the following two simple applications:
\begin{itemize}
\item Note that the lower bound of $N/4$ on the theta-rank of the $N$-gon (cf. Proposition \ref{prop:lbthetarank} in Section \ref{sec:soshierarchy}) can be obtained as a direct corollary of Theorem \ref{thm:clustering}. Indeed if $K$ is contained in the open interval $(-\lceil N/4 \rceil, \lceil N/4 \rceil)$, then the in-diameter of $K$ is $< N/2$ which means that if we consider $K$ as a single cluster, it satisfies conditions (i) and (ii) of the theorem with $\gamma = \indiam(K)$. Thus such a $K$ is not sos-valid.
\item We can also give another simple application of the previous theorem:  Assume $K$ is a set of frequencies that has no two consecutive frequencies, i.e., for any $k,k' \in K$ where $k\neq k'$ we have $d(k,k') \geq 2$. It is not hard to see that such a set $K$ cannot be sos-valid: indeed if $h$ is a function supported on $K$, then the expansion of $|h|^2$ does not have any term involving the frequencies $e_1$ or $e_{-1}$. Thus it is not possible to write $\ell$ as a sum-of-squares of elements supported on such $K$. This simple fact can be obtained as a consequence of Theorem \ref{thm:clustering} if we consider each frequency of $K$ as its own cluster (i.e., we write $K = \cup_{k \in K} \{k\}$) and conditions (i) and (ii) of the theorem are satisfied with $\gamma=1$.
\end{itemize}
\end{rem}

\begin{rem}
\label{rem:mapL}
The map $\cL$ defined in \eqref{eq:defL} in the proof of Theorem \ref{thm:clustering} can actually be interpreted in terms of evaluating a function $h \in \cF(N,\CC)$ at the point $(x=1,y=0)$ ``outside'' the regular $N$-gon. Indeed, note that given any function on the vertices of the $N$-gon $h \in \cF(N,\CC)$, we can extend it naturally to a function (a trigonometric polynomial)  $\tilde{h} \in \CC(z)$ defined on the whole unit circle, as follows:
\begin{equation}
\label{eq:extension}
h = \sum_{k \in \supp(h)} h_k e_k \in \cF(N,\CC) \quad \longmapsto \quad \tilde{h} = \sum_{k \in \supp(h)} h_k e^{-i(k \bmod N) \pi /N} z^{k \bmod N} \in \CC(z).
\end{equation}
Note that \eqref{eq:extension} maps the pure frequencies $e_k$ to the monomial $e^{-i(k\bmod N)\pi/N} z^{k \bmod N} \in \CC(z)$.
Then it is not difficult to verify that $\cL$ satisfies the following points:
\begin{itemize}
\item[(i)] If $\supp h \subseteq B(0,\gamma)$ then $\cL(h) = \tilde{h}(1)$.
\item[(ii)] If $\indiam(\supp(h)) \leq \gamma < N/2$ then $\cL(|h|^2) = |\tilde{h}(1)|^2 \geq
 0$.
\end{itemize}
Point (i) says that if the support of $h$ is contained in the ball centered at $0$ and with radius $\gamma$, then $\cL(h)$ is nothing but the evaluation of the trigonometric polynomial $\tilde{h}$ at $z=1$. This property follows directly from the definitions of $\cL$ and the map \eqref{eq:extension}. Note that it shows in particular that $\cL(\ell) < 0$ (where $\ell$ is the facet functional \eqref{eq:elldef}) since the support of $\ell$ is $\{-1,0,1\}$ and since $\ell$ cuts the point $z=1$ from the regular $N$-gon.
Point (ii) says that if the support of $h$ has in-diameter $\leq \gamma$ then we have $\cL(|h|^2)= |\tilde{h}(1)|^2 \geq 0$. The proof is essentially given in \eqref{eq:eqsii} and uses the fact that if $\indiam(\supp(h)) \leq \gamma$, then $\supp(|h|^2) \subseteq B(0,\gamma)$ and also that $\widetilde{|h|^2} = |\tilde{h}|^2$ (this latter property uses the fact that $\gamma < N/2$).
\end{rem}


\subsection{An algorithm to find valid clusterings and a logarithmic lower bound}

 We now study sets $K$ which admit a clustering that satisfies points (i) and (ii) of Theorem \ref{thm:clustering}. The main purpose of this section is to show that any set $K$ with $|K| < \ln( N/2 )$ admits such a clustering, which implies that it cannot be sos-valid. This would thus show that any $\Rot_N$-equivariant Hermitian psd lift of the regular $N$-gon has to have size at least $\ln( N/2 )$.

For convenience we call a \emph{valid clustering} of a set $K$, any clustering that satisfies points (i) and (ii) of Theorem \ref{thm:clustering}. We state this in the following definition for future reference:
\begin{defn}
\label{def:validclustering}
Let $K \subseteq \Freq{N}$. We say that $K$ has a \emph{valid clustering} if $K$ can be decomposed into disjoint clusters $(C_{\alpha})_{\alpha \in A}$:
\[ K = \bigcup_{\alpha \in A} C_{\alpha}, \]
such that the following holds for some $1\leq \gamma < N/2$:
\begin{enumerate}
\item[(i)] For any $\alpha \in A$, $C_{\alpha}$ has in-diameter $\leq \gamma$.
\item[(ii)] For any $\alpha \neq \alpha'$, $d(C_{\alpha},C_{\alpha'}) > \gamma$.
\end{enumerate}
\end{defn}

We propose a simple greedy algorithm to search for a valid clustering for any set $K \subseteq \Freq{N}$: We start with each point of $K$ in its own cluster and at each iteration we merge the two closest clusters. We keep doing this until we get a clustering that satisfies the required condition, or until all the points are in the same cluster. We show in this section that if the number of points of $K$ is small enough, if $|K| < \ln( N/2 )$, then this algorithm terminates by producing a valid clustering of $K$.
For reference we describe the algorithm more formally in Algorithm \ref{alg:clustering}.
\begin{algorithm}[h!]
\caption{Algorithm to produce a clustering of a set $K$}
\begin{algorithmic}
\STATE {\bf Input:} A set $K \subseteq \Freq{N}$
\STATE {\bf Output:} A \emph{valid clustering} of $K$ (in the sense of Definition \ref{def:validclustering}) or ``0'' if no valid clustering found.
\STATE $\bullet$ Consider initial clustering where each element of $K$ is in its own cluster. If this clustering is already valid (which is equivalent to say that for any distinct elements $k,k' \in K$ we have $d(k,k') \geq 2$) then output this clustering as a valid clustering with parameter $\gamma = 1$.
\STATE $\bullet$ Precompute the pairwise distances between points in $K$ and sort these distances in increasing order $d_1 \leq d_2 \leq d_3 \leq \dots$ (cf. Figure \ref{fig:freqcyclegraph2}).
\FOR{$i=1,2,\dots,|K|-1$}
\STATE Let $x,y \in K$ be the $i$'th closest points in $K$ so that $d(x,y) = d_i$. If $x$ and $y$ are in different clusters, then merge these two clusters.
\STATE If the current clustering satisfies points (i) and (ii) of Definition \ref{def:validclustering} (with $\gamma$ equal to the largest in-diameter in all the clusters) stop and output the current clustering.
\ENDFOR
\STATE If no valid clustering was found, output ``0''
\end{algorithmic}
\label{alg:clustering}
\end{algorithm}

In the next theorem, we show that any set $K \subseteq \Freq{N}$ with $|K| < \ln(N/2)$ has a valid clustering.
\begin{thm}
If a set $K \subseteq \Freq{N}$ satisfies $|K| < \ln(N/2)$, then a valid clustering of $K$ exists and Algorithm \ref{alg:clustering} will produce one.
\end{thm}
\begin{proof}
Observe that at the end of iteration $i$ of the algorithm, the distance between any pair of clusters is greater than or equal $d_{i+1}$: Assume for contradiction that there are two clusters $C,C'$ at iteration $i$ where $d(C,C') < d_{i+1}$. This means that there exist $x \in C$, $y \in C'$ such that $d(x,y) < d_{i+1}$. But this is impossible because the algorithm processes distances in increasing order, and so $x$ and $y$ must have merged in the same cluster at some iteration $\leq i$.

Now, to prove that the algorithm terminates and produces a valid clustering, we need to show that at some iteration $i$, each cluster has in-diameter smaller than $\min(d_{i+1},N/2)$. Note that one can get a simple upper bound on the in-diameter of the clusters at iteration $i$: indeed, it is not hard to show that at iteration $i$ any cluster has in-diameter at most $S_i$, where $S_i$ is defined as:
\[ S_i := d_1+d_2+\dots+d_i = \sum_{j=1}^i d_j. \]
Figure \ref{fig:freqcyclegraph2} shows a simple illustration of this bound.
\newdimen\R
\R=2.0cm
\newdimen\Rs
\Rs=0.1cm
\def\Nnodes{20}
\def\KK{0,1,3,7}
\begin{figure}[ht]
\centering
\begin{tikzpicture}
    \draw[black] circle (\R);
    \foreach \i in {1,...,\Nnodes} {
      \filldraw[fill=white,draw=black] (\i*360/\Nnodes:\R) circle (\Rs);
    }
    \foreach \i in \KK {
      \fill[black] (\i*360/\Nnodes:\R) circle (\Rs);
    }
    \newcommand*{\lasti}{-1}
    \foreach \i [count=\xi] in \KK {
      \ifnum \lasti > -1
		\xdef\as{\lasti*360/\Nnodes}
         \xdef\ae{\i*360/\Nnodes}
         \pgfmathparse{int(\xi-1)}
         \edef\xx{\pgfmathresult}
         \draw[black,<->] ([shift={(\as+1:\R+2*\Rs)}]0,0) arc (\as+1:\ae-1:\R+2*\Rs) node at (.5*\as+.5*\ae:\R+5*\Rs) {$d_{\xx}$};
      \fi
      \xdef\lasti{\i} 
     }
\end{tikzpicture}
\caption{A set of frequencies $K$. At iteration 0 of the algorithm each frequency is in its own cluster. At iteration 1 of the algorithm, the two nodes at distance $d_1$ from each other are merged in a single cluster. At iteration 2, the two nodes at distance $d_2$ are merged and we get one cluster having 3 nodes with in-diameter $d_1+d_2$. In general, at iteration $i$ the clusters cannot have in-diameter larger than $d_1+\dots+d_i$.}
\label{fig:freqcyclegraph2}
\end{figure}
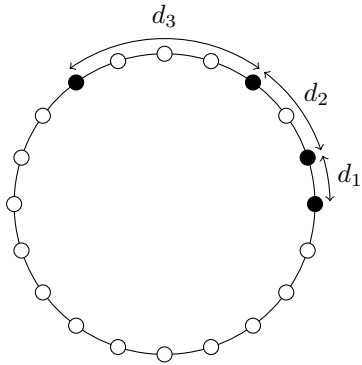

Let $a$ be the largest index $i$ where $d_i = 1$, and let $b$ the largest index $i$ where $S_i < N/2$.\footnote{Note that we can assume $\indiam(K) \geq N/2$ which implies that $S_{|K|-1} \geq N/2$. Indeed, if the in-diameter of $K$ is smaller than $N/2$, then we have a valid clustering of $K$ by considering $K$ as a single cluster.} If $i \in [a,b]$, then at the end of the $i$'th iteration, the distance between any two clusters is greater than 1 (since $d_{i+1} > 1$) and the in-diameter of any cluster is smaller than $N/2$. To prove that the algorithm terminates and produces a valid clustering, it suffices to show that there exists $i \in [a,b]$ such that $d_{i+1} > S_i$.

Assume for contradiction that this is not the case. Then this means that we have:
\[
\begin{aligned}
d_{a+1} & \leq d_1 + \dots + d_a\\
d_{a+2} & \leq d_1 + \dots + d_{a+1}\\
\vdots & \\
d_{b+1} & \leq d_1 + \dots + d_b
\end{aligned}
\]
We will now show that this implies that $|K| \geq \ln( N/2 )$ which contradicts the assumption of the theorem.
Define the function $f(x) = 1/x$ and note that, on the one hand we have:
\[
\sum_{i=a}^b d_{i+1} f(S_i) = \sum_{i=a}^b d_{i+1} \frac{1}{d_1+\dots+d_{i}} \leq \sum_{i=a}^b 1 = b-a + 1.
\]
On the other hand, since $f$ is a decreasing function we have (cf. Figure \ref{fig:integration}):
\[
\sum_{i=a}^b d_{i+1} f(S_i) \geq \int_{S_a}^{S_{b+1}} f(x)dx = \left[\ln(x)\right]_{S_a}^{S_{b+1}} = \ln(S_{b+1}) - \ln(S_a).
\]
Thus we get that:
\[
b-a+1 \geq \ln(S_{b+1}) - \ln(S_a).
\]
Now note that $S_a = a$ since $d_i = 1$ for all $1 \leq i \leq a$. Thus we have:
\[ b \geq \ln(S_{b+1}) - \ln(S_a) + a - 1 \geq \ln(S_{b+1}) \]
since $a - \ln(S_a) \geq 1$ (we assume here that $a \geq 1$ because otherwise the distance between any two elements in $K$ is at least 2 in which case $K$ is clearly not sos-valid).
Now since $|K| \geq b$ and $S_{b+1} \geq N/2$ we get
\[ |K| \geq \ln(N/2) \]
as desired.

\begin{figure}[ht]
  \centering
  \includegraphics[width=9cm]{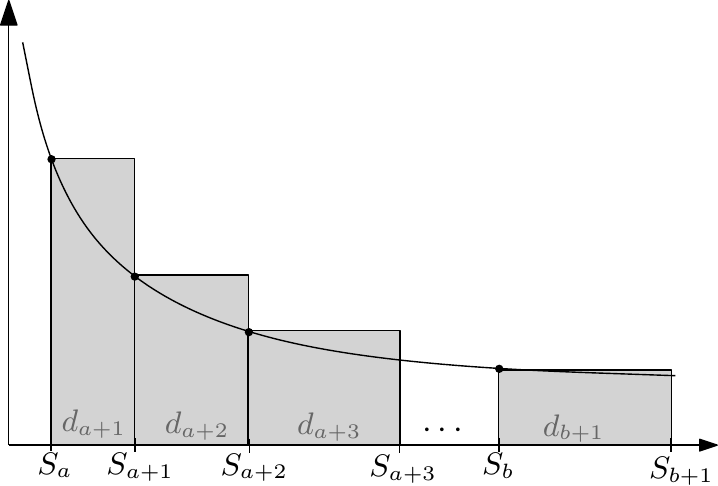}
  \caption{}
  \label{fig:integration}
\end{figure}
\end{proof}

\section{Conclusion}

Regular polygons in the plane have played an important role in the study of extended formulations.
In this paper we studied equivariant psd lifts of regular polygons.
One of the main techniques to obtain equivariant psd lifts of polytopes is using the Lasserre/sum-of-squares hierarchy. The first contribution of this paper was to show that the hierarchy requires exactly $\lceil N/4 \rceil$ iterations for the regular $N$-gon. To prove this we used a specific property about the levels of the facet defining linear functionals of the regular $N$-gon. The techniques we developed are actually quite general and can be used to study the theta-rank of general $k$-level polytopes. For example our techniques allowed us to show that the theta-rank of the parity polytope in $\RR^n$ is $\lceil n/4 \rceil$. They may also be useful
for understanding the theta-rank of other families of $k$-level polytopes, 
such as matroid base polytopes, which were studied in this context 
in the recent work of Grande and Sanyal \cite{grande2014theta}.

The second contribution of this paper was an explicit equivariant psd lift of the regular $2^n$-gon of size $2n-1$. This lift was obtained by showing that the facet-defining linear functionals admit a \emph{sparse} sum-of-squares representation that requires only a small number of ``frequencies''. This construction gives the first example of a polytope with an exponential gap between equivariant psd lifts and equivariant LP lifts. Also it shows that one can construct equivariant psd lifts that are exponentially smaller than the lift produced by the sum-of-squares hierarchy.
We believe that the idea of looking at sparse sum-of-squares representation (i.e., by ``skipping'' frequencies) can potentially be used in other situations and lead to smaller semidefinite lifts.

Finally we proved that the size of our equivariant psd lift is essentially optimal by showing that any equivariant psd lift of the regular $N$-gon has size at least $\ln(N/2)$.
An important question that remains open in the study of regular polygons is to know whether one can obtain smaller psd lifts by relaxing the equivariance condition. Currently the only lower bound on the psd rank of $N$-gons in the plane is $\Omega\left(\sqrt{\frac{\log N}{\log \log N}}\right)$ which comes from quantifier elimination theory \cite{gouveia2011lifts,gouveia2013worst}.


%


\appendix

\newpage

\section{Finishing the proof on theta-rank of the regular $N$-gon}
\label{app:theta-rank}

In this appendix we complete the proof of Theorem \ref{thm:theta-rank-Ngon} concerning the theta-rank of the $N$-gon. We first prove the following lemma:
\begin{lem}
\label{lem:chebyshev-tangent}
Let $N$ be a positive integer and let $T_N$ be the Chebyshev polynomial of degree $N$. Then for any $u \geq \cos(\pi/N)$, the curve of $T_N(x)$ lies above its tangent at $x=u$ on the interval $[-1,\infty)$, i.e.,
\begin{equation}
\label{eq:ineqTN}
 T_N(x) \geq T_N(u) + T'_N(u)(x - u) \quad \forall x \in [-1,\infty).
\end{equation}
Furthermore, when $N$ is even the inequality \eqref{eq:ineqTN} is true for all $x \in \RR$.
\end{lem}
An illustration of Lemma \ref{lem:chebyshev-tangent} is given in Figure \ref{fig:plot_chebyshev_lemma_appendix}.
\begin{proof}
First observe that $T_N''(x) \geq 0$ for all $x \in [\cos(\pi/N),\infty)$: indeed note that $\cos(\pi/N)$ is the largest root of $T'_N$, and thus, since the roots of $T''_N$ interlace the roots of $T'_N$ we have necessarily that $T''_N \geq 0$ on $[\cos(\pi/N),\infty)$. Thus this shows that $T_N$ is convex on the interval $[\cos(\pi/N),\infty)$ and in particular shows that inequality \eqref{eq:ineqTN} holds for all $x \in [\cos(\pi/N),\infty)$.
It remains to show that the inequality \eqref{eq:ineqTN} holds for $x \in [-1,\cos(\pi/N))$. Since $\cos(\pi/N)$ is a minimum of $T_N$ on the interval $[-1,1]$ we have, for any $x \in [-1,\cos(\pi/N))$:
\[ T_N(x) \geq T_N(\cos(\pi/N)) \overset{(a)}{\geq} T_N(u) + T'_N(u) (\cos(\pi/N) - u) \overset{(b)}{\geq} T_N(u) + T'_N(u) (x - u) \]
where $(a)$ follows from the first part of the argument which shows that inequality \eqref{eq:ineqTN} holds for $x=\cos(\pi/N)$ and, where in $(b)$ we used the fact that $x \leq \cos(\pi/N)$ and that $T'_N(u) \geq 0$. Thus this proves inequality \eqref{eq:ineqTN}.

When $N$ is even inequality \eqref{eq:ineqTN} is clearly true for $x \leq -1$ also since for $x \leq -1$, $T_N(x) \geq 0$ whereas the linear function $T_N(u) + T'_N(u) (x - u)$ is negative.
\end{proof}

\begin{figure}[ht]
  \centering
  \includegraphics[width=9cm]{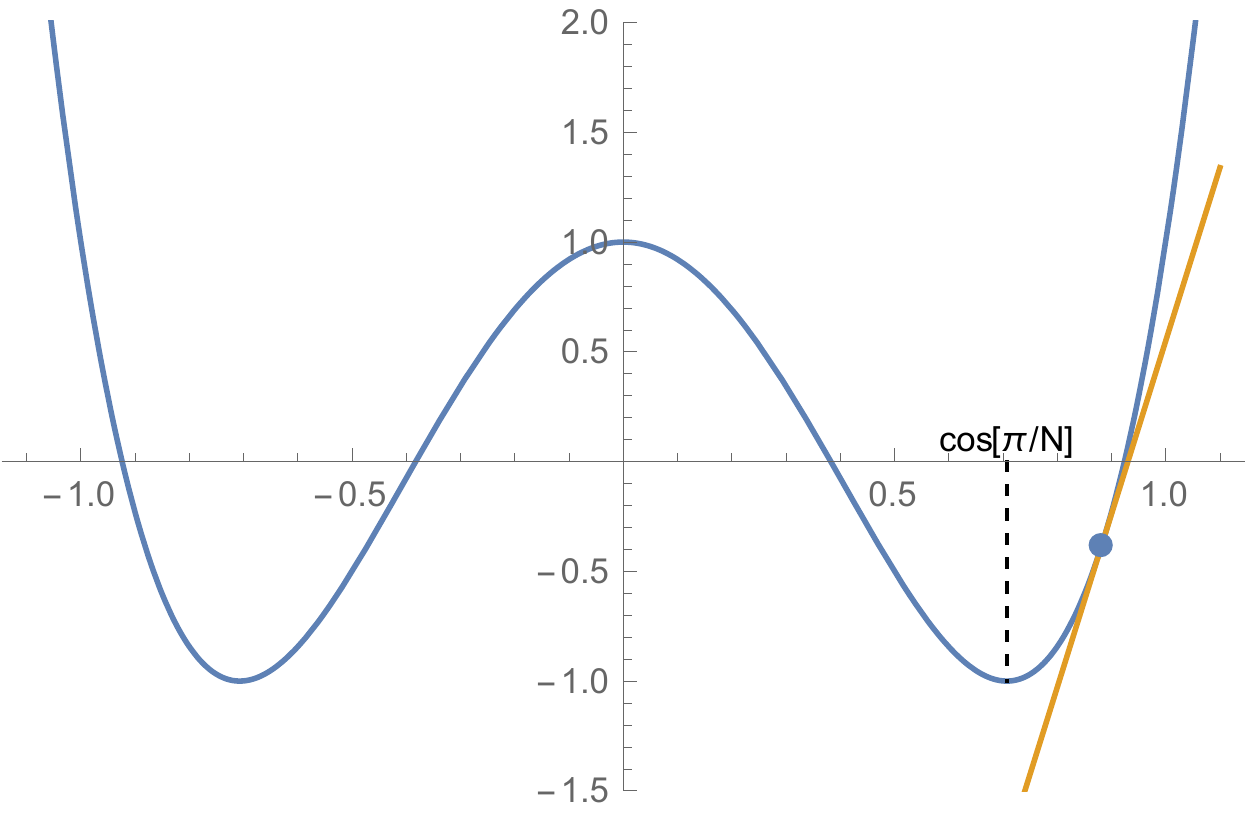}
  \caption{Illustration of Lemma \ref{lem:chebyshev-tangent} with $N=4$ and some value $u \geq \cos(\pi/N)$.}
  \label{fig:plot_chebyshev_lemma_appendix}
\end{figure}

We now complete the proof of Theorem \ref{thm:theta-rank-Ngon} by considering the cases where $N$ is not a necessarily a multiple of four. For $i=0,\dots,\lceil N/2 \rceil-1$, let $a_i = x_{i+1,N} = \cos((2i+1)\pi/N)$ and let $q_N$ be the polynomial that vanishes at the $a_i$'s:
\begin{equation}
\label{eq:qN}
q_N(x) = \prod_{i=0}^{\lceil N/2 \rceil-1} (x - a_i) = \prod_{i=0}^{\lceil N/2 \rceil-1} \left(x - \cos\left(\frac{(2i+1)\pi}{N}\right)\right).
\end{equation}
In the case where $N$ is a multiple of four we saw that $q_N(x)$ is, up to a scalar, $T_{N/2}(x)$. The next lemma expresses the polynomial $q_N$ in terms of Chebyshev polynomials for any $N$:
\begin{lem}
\label{lem:qN-Cheb}
The polynomial $q_N$ satisfies:
\[ q_N(x) \propto \begin{cases} T_{N/2}(x) & \text{ if $N$ is even}\\
(T_{\lfloor N/2\rfloor}(x) + T_{\lceil N/2\rceil}(x))/2 & \text{ if $N$ is odd}.
\end{cases} \]
where the symbol $\propto$ indicates equality up to multiplicative constant.
\end{lem}
\begin{proof}
The case where $N$ is even is clear by comparing the roots of $q_N$ and those of $T_{N/2}$. For the case $N$ odd, observe that if $\cos\alpha$ is a root of $q_N$ then $\pm \cos( \alpha/2)$ are roots of $q_N(T_2(x)) = q_N(2x^2-1)$. Since the roots of $q_N$ are the $\{\cos((2i-1)\pi/N),i=1,\dots,\lceil N/2 \rceil\}$, the roots of $q_N(2x^2-1)$ are thus  $\{\pm \cos((2i-1)\pi/(2N)), i=1,\dots,\lceil N/2 \rceil\}$ (with a double root at 0). Note that these are exactly the roots of $xT_N(x)$ (the multiplication by $x$ is for the double root at 0). Thus from this observation we have for any $x \in \RR$:
\[
\begin{aligned} q_N(T_2(x)) \propto xT_N(x) = T_1(x) T_N(x) &\overset{(a)}{=} (T_{N-1}(x) + T_{N+1}(x))/2\\
&\overset{(b)}{=} (T_{(N-1)/2}(T_2(x)) + T_{(N+1)/2}(T_2(x)))/2.
\end{aligned} \]
Equality $(a)$ follows from the identity $T_a(x) T_b(x) = \frac{1}{2} (T_{a+b}(x) + T_{a-b}(x))$ and equality $(b)$ follows from $T_a(T_b(x)) = T_{ab}(x)$.
Thus since we are working with polynomials and since $\{T_2(x): x \in \RR\}$ is infinite we have the desired identity:
\[ q_N(x) \propto T_{\lfloor N/2\rfloor}(x) + T_{\lceil N/2\rceil}(x). \]
\end{proof}

We are now ready to finish the proof of Theorem \ref{thm:theta-rank-Ngon}. We distinguish the three remaining cases according to the residue class of $N$ modulo 4:
\begin{itemize}
\item Case $N=4m-1$: In this case the polynomial $q_N$ is even degree and we want to show that $q_N(x)$ is above its linear approximation at $x=\cos(\pi/N)$. From Lemma \ref{lem:qN-Cheb} we have that $q_N(x) \propto T_{2m-1}(x) + T_{2m}(x)$.    Since, for all $x\in [-1,\infty)$, $T_{2m}(x)$ and $T_{2m-1}(x)$ are both above their linear approximations at $\cos(\pi/N)$ (using Lemma \ref{lem:chebyshev-tangent} and because
    $\cos(\pi/N) \geq \cos(\pi/(2m))$ and $\cos(\pi/N) \geq \cos(\pi/(2m-1))$) it follows that the same holds for 
    $q_N$ on $[-1,\infty)$. Since, in addition $q_N$ has even degree and $q_N(-1) \geq 0$ this shows that $q_N(x)$ is above its linear approximation at $\cos(\pi/N)$ for all $x$, which is what we wanted.

\item Case $N = 4m-2$: In this case the polynomial $q_N$ is $q_{N}(x) = T_{2m-1}(x)$ (from Lemma \ref{lem:qN-Cheb}). Note that $q_N$ has odd degree. Thus to apply Proposition \ref{prop:interp2} we will add an additional ``dummy'' root for $q$ to make it even degree (the resulting interpolating polynomial $p$ we get will interpolate the linear function $l$ at this additional ``dummy'' point).
Consider the polynomial $\tilde{q_N}(x) = x q_N(x)$. We will show that the assumption of Proposition \ref{prop:interp2} holds for $\tilde{q_N}(x)$. Observe that
\[ \tilde{q_N}(x) = xq_N(x) \propto T_1(x) T_{2m-1}(x) = (T_{2m}(x)+T_{2m-2}(x))/2. \]
    Since both $T_{2m}$ and $T_{2m-2}$ are globally above their linear approximations at $\cos(\pi/N)$ (by Lemma \ref{lem:chebyshev-tangent} and because $\cos(\pi/N) \geq \cos(\pi/(2m))$
    and $\cos(\pi/N) \geq \cos(\pi/(2m-1))$), the same holds for $\tilde{q_N}(x) = x q_N(x)$. Thus this shows that $\tilde{q_N}(x)$ lies above its tangent at $x=\cos(\pi/N)$, which is what we want.
\item Case $N=4m-3$: In this case we have, from Lemma \ref{lem:qN-Cheb}, $q_N(x) = (T_{2m-1}(x)+T_{2m-2}(x))/2$. Note that the polynomial $q_N$ has odd degree and thus we need to add an additional ``dummy'' root to make it even degree. Take $\tilde{q_N}(x) = xq_N(x)$ and note that
 \[ \tilde{q_N}(x) \propto (T_{2m}(x)+T_{2m-2}(x) + T_{2m-1}(x) + T_{2m-3}(x))/4. \]
    Using Lemma \ref{lem:chebyshev-tangent}, for $x\in [-1,\infty)$, each of the four Chebyshev polynomials are above their linear approximation at $\cos(\pi/N)$ (because 
    $\cos(\pi/N) \geq \cos(\pi/(2m))$ and $\cos(\pi/N) \geq \cos(\pi/(2m-1))$ and $\cos(\pi/N) \geq \cos(\pi/(2m-2))$ and $\cos(\pi/N) \geq \cos(\pi/(2m-3))$).
    Since $\tilde{q_N}(x)$ has even degree and $\tilde{q_N}(-1) \geq 0$, it holds that $\tilde{q_N}$ is globally above its linear approximation at $x=\cos(\pi/N)$.
\end{itemize}

\section{Linear programming lifts}
\label{app:equivariantLPlifts}

In this section we recall the definitions of LP lifts and equivariant LP lifts. For reference we also provide the proof from \cite{gouveia2011lifts} that any equivariant LP lift of the regular $N$-gon must have size at least $N$ when $N$ is a power of a prime.

We first recall the definition of a linear programming (LP) lift:
\begin{defn}
\label{def:LPlift}
Let $P \subset \RR^n$ be a polytope. We say that $P$ has a LP lift of size $d$ if we can write $P = \pi(\RR^d_+ \cap L)$ where $\pi:\RR^d \rightarrow \RR^n$ is a linear map and $L$ is an affine subspace of $\RR^d$.
\end{defn}
We now give the definition of an equivariant LP lift, from \cite{gouveia2011lifts,kaibel2010symmetry} (also known as symmetric LP lift). We denote by $\mathfrak{S}_d$ the group of permutations on $d$ elements. If $\sigma \in \mathfrak{S}_d$ and $y \in \RR^d$, we denote by $\sigma y$ the left action of $\mathfrak{S}_d$ on $\RR^d$ which permutes the coordinates according to $\sigma$.
\begin{defn}
\label{def:equivariantLPlift}
Let $P \subset \RR^n$ be a polytope and assume that $P$ is invariant under the action of a group $G$. Let $P = \pi(\RR^d_+ \cap L)$ be a LP lift of size $d$. The lift is called $G$-equivariant if there exists a homomorphism $\Phi:G \rightarrow \mathfrak{S}_d$ such that:
\begin{itemize}
\item[(i)] The subspace $L$ is invariant under the permutation action of $\Phi(g)$, for all $g \in G$:
\begin{equation}
  \label{eq:LP-Linvariance}
  \Phi(g) y \in L \quad \forall g \in G, \; \forall y \in L.
\end{equation}
\item[(ii)] The following equivariance relation holds:
\begin{equation}
 \label{eq:LP-linequivariance}
 \pi(\Phi(g) y) = g\pi(y) \quad \forall g \in G, \; \forall y \in \RR^d_+ \cap L.
\end{equation}
\end{itemize}
\end{defn}

Given integer $N$, let $\Rot_N$ be the subgroup of rotations of the dihedral group of order $2N$. Note that $\Rot_N \isom \ZZ_N$.
\begin{prop} \cite[Proposition 3]{gouveia2011lifts}
If $N$ is a prime or a power of a prime, then any $\Rot_N$-equivariant LP lift of the regular $N$-gon has size $N$.
\end{prop}
\begin{proof}
Let $P$ be the regular $N$-gon and assume that $P$ has a LP lift of size $d$ that is $\Rot_N$-equivariant. By Definition \ref{def:equivariantLPlift} there exists a homomorphism $\Phi:\Rot_N \rightarrow \mathfrak{S}_d$ such that \eqref{eq:LP-Linvariance} and \eqref{eq:LP-linequivariance} are satisfied. It is not hard to show that $\Phi$ must be injective: indeed if $\Phi(g)=1$ for some $g \in \Rot_N$ then by the equivariance relation \eqref{eq:LP-linequivariance} we must have $\pi(y) = g\pi(y)$ for all $y \in \RR^d_+ \cap L$, which means that $x = g x$ for all $x \in P$. Since $P$ is full-dimensional this means that $g$ is the identity element in $\Rot_N$.

Since $\Phi$ is injective, we have that $\Phi(\Rot_N)$ is a cyclic subgroup of $\mathfrak{S}_d$ of size $N$ and thus $\mathfrak{S}_d$ has an element of order $N$. One can show that if $\mathfrak{S}_d$ has an element of order $p^t$ where $p$ is a prime and $t \geq 1$, then $d \geq p^t$: to see this one can use the decomposition of a permutation into cycles with disjoint support, and recall that the order of a permutation is the least common multiple of the cycle lengths; thus if the order of a permutation is $p^t$ then at least one of the cycle lengths must be divisible by $p^t$ which implies that $d \geq p^t$. Thus this shows that when $N$ has the form $N = p^t$ then we must have $d \geq p^t$.
\begin{rem}
When $N$ is a prime we easily see that we must have $d \geq N$ by the simple fact that $N$ is a prime and that it has to divide $d!$.
\end{rem}
\end{proof}

\bibliography{../../../bib/nonnegative_rank}
\bibliographystyle{alpha}

\end{document}